\documentclass[10pt]{amsart}
\usepackage{cm-math}
\usepackage{cm-style}

\usepackage[a4paper]{geometry}
\makeatletter
\renewcommand\section{\@startsection{section}{1}%
  \z@{.7\linespacing\@plus\linespacing}{.5\linespacing}%
  {\Large\rmfamily\scshape\centering}}
\makeatother

\begin{document}
\title[Multilevel quadrature for elliptic problems on random domains]
{Multilevel quadrature\\ 
for elliptic problems on random domains\\
by the coupling of FEM and BEM}
\author{Helmut~Harbrecht}
\author{Marc Schmidlin}
\address{
Helmut Harbrecht and Marc Schmidlin,
Universit\"at Basel,
Departement Mathematik und Informatik,
Spiegelgasse 1, 4051 Basel, Schweiz
}
\email{\{helmut.harbrecht,marc.schmidlin\}@unibas.ch}
\thanks{The authors gratefully acknowledge the 
support from the Swiss National Science Foundation 
(Grant No.\ 205321\_169599).}
\date{}

\begin{abstract}
Elliptic boundary value problems which are posed
on a random domain can be mapped to a fixed, nominal domain.
The randomness is thus transferred to the diffusion
matrix and the loading.
While this domain mapping method is quite efficient for theory and practice,
since only a single domain discretisation is needed,
it also requires the knowledge of the domain mapping.

However,
in certain applications, the random domain is only
described by its random boundary,
while the quantity of interest is defined on a fixed, deterministic subdomain.
In this setting,
it thus becomes necessary to compute a random domain mapping
on the whole domain,
such that the domain mapping is the identity on the fixed subdomain
and maps the boundary of the chosen fixed, nominal domain on to the random boundary.

To overcome the necessity of computing such a mapping,
we therefore couple the finite element method on the fixed subdomain
with the boundary element method on the random boundary.
We verify the required regularity of the solution
with respect to the random domain mapping for
the use of multilevel quadrature,
derive the coupling formulation,
and show by numerical results that the approach is feasible.
\end{abstract}

\maketitle

\section{Introduction}\label{section:introduction}
Many practical problems in science and engineering lead
to elliptic boundary value problems for an unknown function.
Their numerical treatment by e.g.\ finite difference
or finite element methods is in general well understood provided 
that the input parameters are given exactly.
This, however, is often not the case in practical applications. 

If a statistical description of the input data is available, 
one can mathematically describe data and solutions as 
random fields and aim at the computation of corresponding 
deterministic statistics of the unknown random solution. 
The present article is dedicated to the treatment of uncertainties 
in the description of the computational domain.
Applications are, besides traditional engineering,
for example uncertain domains which are derived from inverse methods
such as tomography.
In recent years, this situation has become of growing interest,
see e.g.~\cite{CK07,CNT16,HPS16,HSS08,MNK11,TX06,XT06} and the references therein. 

In this article, we are first going to focus on the so-called domain mapping method,
which has been introduced in \cite{XT06} and rigorously analysed in \cite{CNT16,HPS16},
where analytic dependency of the solution on the random domain mapping
with regard to the energy norm has been verified.
Given enough spatial regularity of the random domain mapping,
we first prove that the solution is analytically dependent on the random domain mapping
also in the $H^s(D)$-norm.
The key idea of the method is to map the boundary value problem 
\begin{equation}\label{eq:SPDE1}
  -\Laplace_\xbfm u[\omega] = f \text{ in $\Dfrak[\omega]$}, \quad
  u[\omega] = 0 \text{ on $\partial \Dfrak[\omega]$},
\end{equation}
which is posed on a random domain
\begin{equation*}
  \Dfrak[\omega] \isdef \Vbfm[\omega](D) \subset \Rbbb^d
\end{equation*}
onto a fixed, nominal \emph{reference} domain $D \subset \Rbbb^d$.
Thus, the randomness is transferred to the diffusion matrix
and the loading of the boundary value problem
\begin{equation}\label{eq:SPDE2}
  -\Div_\xbfm\groupp[\big]{\hat{\Abfm}[\omega] \Grad_\xbfm \hat{u}[\omega]} 
  = \hat{f}[\omega] \text{ in $D$}, \quad
  \hat{u}[\omega] = 0 \text{ on $\partial D$}.
\end{equation}
Herein, it holds
\begin{equation}\label{eq:pullbackcoefficientdata}
  \hat{\Abfm}[\omega]
  \isdef \groupp[\big]{\Jbfm[\omega]^\trans \Jbfm[\omega]}^{-1} \det\Jbfm[\omega]
  \quad\text{and}\quad
  \hat{f}[\omega]
  \isdef \groupp[\big]{f \circ \Vbfm[\omega]} \det\Jbfm[\omega],
\end{equation}
where $\Jbfm[\omega]$ denotes the Jacobian of the field
$\Vbfm[\omega] \colon D \to \Dfrak[\omega]$
\begin{equation}\label{eq:jacobian}
  \Jbfm[\omega](\xbfm) \isdef \Dif_\xbfm \Vbfm[\omega](\xbfm) .
\end{equation}
and $\hat{u}[\omega]$ is connected to $u[\omega]$ by
$\hat{u}[\omega] \isdef \groupp[\big]{u[\omega] \circ \Vbfm[\omega]}$.

While the random domain mapping approach is mathematically natural,
it is not neccessarily the setting,
which is directly encountered in practical applications.
This mainly stems from the fact that the random domain mapping
does not only describe the random domains themselves
but also includes a specific point correspondence between the domain realisations.
In applications often only a description of the random boundary might be known,
however in such cases the quantity of interest 
\begin{equation}\label{eq:QoI}
  \QoI(u) = \int_\Omega \Fcal\groupp[\big]{u[\omega]\vert_B} \dif\!\Pbbb[\omega]
\end{equation}
is generally sought on a \emph{deterministic} subdomain, $B$,
which almost surely is a subset of the domain realisations.
Therefore,
it is then necessary to be able to transform the description of the random domains
given by a description of the random boundary and the specification of the subdomain
into the form of a random domain mapping.
In \cite{XT06}, the authors consider by using the vector\hyp{}valued Laplace equation
to compute such a random domain mapping. If more structure is given, for 
example when the random domains are described by star-shaped boundaries
or more generally when they are directly given by a boundary mapping from 
a nominal boundary, one may also consider other approaches, such as 
transfinite interpolation techniques, see e.g.~\cite{G71,GH73,GT82},
to extend the mapping onto the whole reference domain.

However,
to overcome the necessity of computing such a random domain mapping in this setting,
we propose to compute the quantity of interest by performing the calculations
on the realisations of the random domains.
Moreover,
we can also sidestep the generation of a mesh on the random part of the domain
$\Dfrak[\omega] \setminus B$,
by coupling finite element methods with 
boundary element methods for the spatial approximation as follows:
we apply finite elements on the subdomain $B$ 
and treat the rest of the domain by a boundary element method.
This is advantageous,
since large domain deformations on coarse discretisations can be handled easily,
as we do not need to mesh the random part of the domain but only its boundary.
We present the resulting coupling formulation and 
then discuss the efficient solution by multilevel quadrature methods.
Especially, since we verify the required regularity with respect to the random
perturbation field,
we also know that we have the required regularity on the deterministic subdomain $B$,
under the assumption that there exists a transform from the random boundary description
to the random domain mapping,
which has sufficient regularity.

The rest of this article is organized as follows. Section
\ref{section:modelproblem} is dedicated to the mathematical 
formulation of the problem under consideration. The problem's 
regularity is studied in Section~\ref{section:regularity}. Here, we
provide estimates in stronger spatial norms which are needed 
for multilevel accelerated quadrature methods. The coupling of 
finite elements and boundary elements is the topic of Section
\ref{section:coupling}. The multilevel quadrature method for the
solution of the random boundary value problem is then introduced
in Section~\ref{section:MLQ}. Numerical experiments are carried 
out in Section~\ref{section:results}. Finally, we state concluding 
remarks in Section~\ref{section:conclusion}.

\section{Notation and model problem}\label{section:modelproblem}
%
%
Before we complete the mathematical setting of our model problem,
we will introduce the notations used throughout the rest of the article.
Especially, for the regularity considerations in Section~\ref{section:regularity}
some of the notation --- and the choice of a certain weighting in the
Sobolev--Bochner norms --- helps keep formulas somewhat more concise and compact.
\subsection{Notation and precursory remarks}
We use $\Nbbb$ to denote the natural numbers including $0$
and $\Nbbb^*$ when excluding $0$.

For a sequence of natural numbers,
$\alphabfm = \groupb{\alpha_n}_{n \in \Nbbb^*} \in \Nbbb^{\Nbbb^*}$,
we define the support of the sequence as
\begin{equation*}
  \supp \alphabfm = \groupb{n \in \Nbbb^* \,|\, \alpha_n \neq 0}
\end{equation*}
and say that $\alphabfm$ is finitely supported,
if $\supp \alphabfm$ is of finite cardinality,
Then,
$\Nbbb^{\Nbbb^*}_f$ denotes the set of finitely supported sequences of natural numbers
and we refer to its elements as multi\hyp{}indices.
Furthermore,
for all $m \in \Nbbb^*$ we will identify the elements
$\alphabfm = \groupp{\alpha_1, \ldots, \alpha_m} \in \Nbbb^m$
with their extension by zero into $\Nbbb^{\Nbbb^*}_f$,
that is $\alphabfm = \groupp{\alpha_1, \ldots, \alpha_m, 0, \ldots}$.
Thus,
by this identification,
all notations defined for elements of $\Nbbb^{\Nbbb^*}_f$
also carry over to the elements of $\Nbbb^m$
and we also refer to elements of $\Nbbb^m$ as multi\hyp{}indices.

For multi\hyp{}indices
$\alphabfm = \groupb{\alpha_n}_{n \in \Nbbb^*}, \betabfm = \groupb{\beta_n}_{n \in \Nbbb^*} \in \Nbbb^{\Nbbb^*}_f$
and a sequence of real numbers $\gammabfm = \groupb{\gamma_n}_{n \in \Nbbb^*} \in \Rbbb^{\Nbbb^*}$,
we use the following common notations:
\begin{align*}
  \norms{\alphabfm} &\isdef \sum_{n \in \supp \alphabfm} \alpha_n , &
  \alphabfm! &\isdef \prod_{n \in \supp \alphabfm} \alpha_n! , \\
  \binom{\alphabfm}{\betabfm} &\isdef \prod_{n \in \supp \alphabfm \cup \supp \betabfm} \binom{\alpha_n}{\beta_n} , &
  \gammabfm^\alphabfm &\isdef \prod_{n \in \supp \alphabfm} \gamma_n^{\alpha_n} .
\end{align*}
Furthermore,
we say that $\alphabfm \leq \betabfm$ holds,
when $\alpha_j \leq \beta_j$ holds for all $j \in \supp \alphabfm \cup \supp \betabfm$,
and $\alphabfm < \betabfm$, when $\alphabfm \leq \betabfm$ and $\alphabfm \neq \betabfm$ hold.

Subsequently,
we will always equip $\Rbbb^m$ with the norm $\norm{\cdot}_2$
induced by the canonical inner product $\langle\cdot,\cdot\rangle$
and $\Rbbb^{m \times m}$ with the induced norm $\norm{\cdot}_2$.
Moreover, when considering $\Rbbb^m$ itself or an open domain $\Dcal \subset \Rbbb^m$
as a measure space we always equip it with the Lebesgue measure.
Similarly,
we always equip $\Nbbb$ and $\Nbbb^*$ with the counting measure,
when considering them as measure spaces.

Let $\Xcal$, $\Xcal_1, \ldots, \Xcal_r$ and $\Ycal$ be Banach spaces,
then we denote the Banach space of bounded, linear maps from $\Xcal$ to $\Ycal$
as $\Bcal(\Xcal; \Ycal)$;
furthermore,
we recursively define
\begin{equation*}
  \Bcal(\Xcal_1, \ldots, \Xcal_r; \Ycal) \isdef \Bcal\groupp[\big]{\Xcal_1; \Bcal(\Xcal_2, \ldots, \Xcal_r; \Ycal)}
\end{equation*}
and the special case
\begin{equation*}
  \Bcal^0(\Xcal; \Ycal) \isdef \Ycal
  \quad\text{and}\quad
  \Bcal^{r+1}(\Xcal; \Ycal) \isdef \Bcal\groupp[\big]{\Xcal; \Bcal^r(\Xcal; \Ycal)} .
\end{equation*}
For $\Tbfm \in \Bcal(\Xcal_1, \ldots, \Xcal_r; \Ycal)$ and $\vbfm_j \in \Xcal_j$
we use the shorthand notation
$\Tbfm \vbfm_1 \cdots \vbfm_r \isdef \Tbfm(\vbfm_1, \ldots, \vbfm_r) \in \Ycal$.

For a given Banach space $\Xcal$ and a complete measure 
space $\Mcal$ with measure $\mu$ the space $L_\mu^p(\Mcal; \Xcal)$ 
for $1 \leq p \leq \infty$ denotes the Bochner space, see \cite{HillePhillips},
which contains all equivalence classes of strongly measurable functions
$v \colon \Mcal \to \Xcal$ with finite norm
\begin{equation*}
  \norm{v}_{p, \Mcal; \Xcal}
  \isdef 
  \norm{v}_{L_\mu^p(\Mcal; \Xcal)} \isdef
  \begin{cases}
    \groups[\bigg]{\displaystyle\int_\Mcal \norm[\big]{v(x)}_{\Xcal}^p \dif\,\mu(x)}^{1/p} , & p < \infty , \\
    \displaystyle\esssup_{x \in \Mcal} \norm[\big]{v(x)}_{\Xcal} , & p = \infty .
  \end{cases}
\end{equation*}
A function $v \colon \Mcal \to \Xcal$ is strongly measurable if there exists 
a sequence of countably\hyp{}valued measurable functions $v_n \colon \Mcal \to \Xcal$,
such that for almost every $m \in \Mcal$ we have $\lim_{n \to \infty} v_n(m) = v(m)$.
Note that, for finite measures $\mu$, we also have the usual inclusion
$L_\mu^p(\Mcal; \Xcal) \supset L_\mu^q(\Mcal; \Xcal)$ for $1 \leq p < q \leq \infty$.

For a given Banach space $\Xcal$ and an open domain $\Dcal \subset \Rbbb^d$,
with $d \in \Nbbb^*$,
the space $W^{\eta, p}(\Dcal; \Xcal)$
for $\eta \in \Nbbb$ and $1 \leq p \leq \infty$ denotes the Sobolev--Bochner space,
which contains all equivalence classes of strongly measurable functions
$v \colon \Dcal \to \Xcal$,
such that the function itself and all weak derivatives up to total order $\eta$
are in $L^p(\Dcal; \Xcal)$ with the norm
\begin{equation*}
  \norm{v}_{\eta, p, \Dcal; \Xcal}
  \isdef \norm{v}_{W^{\eta, p}(\Dcal; \Xcal)}
  \isdef
  \sum_{\norms{\alphabfm} \leq \eta} \frac{1}{\alphabfm!}
  \norm[\big]{\pdif_\xbfm^\alphabfm v}_{p, \Dcal; \Xcal} .
\end{equation*}
Moreover,
$W_0^{\eta, p}(\Dcal; \Xcal)$ denotes the closure of the linear subspace of
smooth functions with compact support, $C_c^\infty(\Dcal; \Xcal)$, in $W^{\eta, p}(\Dcal; \Xcal)$
and we set $H^\eta(\Dcal; \Xcal) \isdef W^{\eta, 2}(\Dcal; \Xcal)$
and $H_0^\eta(\Dcal; \Xcal) \isdef W_0^{\eta, 2}(\Dcal; \Xcal)$.
As usual,
we use $C^{\omega}(\Dcal; \Xcal)$ to denote the real analytic functions from $\Dcal$ to $\Xcal$
and $C^{k, s}(\Dcal; \Xcal)$ to denote the H\"older spaces.
For a bi\hyp{}Lipschitz function $v \colon \Dcal \to \Xcal$ we denote its bi\hyp{}Lipschitz constants by
\begin{align*}
  \norms{v}_{\Lipl(\Dcal; \Xcal)}
  &\isdef \essinf_{\xbfm, \ybfm \in \Dcal,\, \xbfm \neq \ybfm} \frac{\norm{v(\xbfm) - v(\ybfm)}_{\Xcal}}{\norm{\xbfm - \ybfm}}, \\
  \norms{v}_{\Lipu(\Dcal; \Xcal)}
  &\isdef \esssup_{\xbfm, \ybfm \in \Dcal,\, \xbfm \neq \ybfm} \frac{\norm{v(\xbfm) - v(\ybfm)}_{\Xcal}}{\norm{\xbfm - \ybfm}} .
\end{align*}

In the notation for the Bochner, Sobolev--Bochner and H\"older spaces, 
we may omit specifying the Banach space $\Xcal$ when $\Xcal = \Rbbb$.
Especially, $H^{-\eta}(\Dcal)$ denotes the topological dual space of $H_0^\eta(\Dcal)$.
Moreover,
if the $\Xcal$ we are considering is itself a Bochner or Sobolev--Bochner space,
then we replace the $\Xcal$ in the subscript of the norm with the subscripts of its norm,
for example
\begin{equation*}
  \norm{v}_{p, \Mcal; \eta, q, \Dcal; \Ycal}
  = \norm{v}_{p, \Mcal; W^{\eta, q}(\Dcal; \Ycal)}
  = \norm{v}_{L_\mu^p(\Mcal; W^{\eta, q}(\Dcal; \Ycal))} .
\end{equation*}

Lastly,
to avoid the use of generic but unspecified constants in
certain formulas, we use $c \lesssim d$ to mean that $c$ can be bounded by
a multiple of $d$, independently of parameters which $c$ and $d$ may depend on.
Obviously, $c \gtrsim d$ is defined as $d \lesssim c$ and we write
$c \eqsim d$ if $c \lesssim d$ and $c \gtrsim d$.
\subsection{Model problem}
Let $\tau \in \Nbbb$ and $d \in \Nbbb^\ast$;
$D \subset \Rbbb^d$ denote the \emph{reference domain}
with boundary $\partial D$ that is of class $C^{\tau, 1}$ ---
when $\tau = 1$ then we also consider the case where $D$ is
a bounded and convex domain with Lipschitz continuous boundary ---
and $(\Omega, \Fcal, \Pbbb)$ be a separable, complete probability space
with $\sigma$-field $\Fcal \subset 2^\Omega$ and probability measure $\Pbbb$.
Furthermore, let
\begin{equation*}
  \Vbfm \in L_\Pbbb^\infty\groupp[\big]{\Omega; C^{\tau, 1}(\overline{D}; \Rbbb^d)}
\end{equation*}
be the \emph{random domain mapping}.
Moreover, we require that,
for $\Pbbb$-almost any $\omega$,
$\Vbfm[\omega] \colon D \to \Dfrak[\omega]$ is bi\hyp{}Lipschitz
and fulfils the \emph{uniformity condition}
\begin{equation*}
  \underline{\sigma}
  \leq \norms[\big]{\Vbfm[\omega]}_{\Lipl(D; \Rbbb^d)}
  \leq \norms[\big]{\Vbfm[\omega]}_{\Lipu(D; \Rbbb^d)}
  \leq \overline{\sigma}
\end{equation*}
for $0 < \underline{\sigma} \leq \overline{\sigma} < \infty$ independent of $\omega$.
Finally, we require that the we have a \emph{hold-all domain} $\Dcal$
that satisfies $\Dfrak[\omega] \subset \Dcal$ for $\Pbbb$-almost any $\omega \in \Omega$
and consider $f \in C^\omega(\Dcal)$.

Note that while we restrict ourselves to the Poisson equation here to simplify
the analysis, the extension of the regularity result
to an operator $\Div_\xbfm \Abfm \Grad_\xbfm$,
with an $\Abfm \in C^\omega(\Dcal; \Rbbb^{d \times d})$ and $\Abfm$ fulfilling
an ellipticity condition is straightforward.

While, by definition,
we know that $\Vbfm[\omega]$ is a $C^{0,1}$-diffeomorphism from $D \to \Dfrak[\omega]$
for $\Pbbb$-almost any $\omega \in \Omega$,
we also have the following stronger result.
\begin{proposition}
  For $\Pbbb$-almost any $\omega \in \Omega$,
  $\Vbfm[\omega]$ is a $C^{\tau,1}$-diffeomorphism from $D$ to $\Dfrak[\omega]$.
\end{proposition}
\begin{proof}
  The fact that $\Vbfm[\omega]$ is a $C^{\tau}$-diffeomorphism follows directly
  from the inverse funtion theorem.
  Then,
  with the explicit formula for the $\tau$\hyp{}th derivative of $\Vbfm[\omega]^{-1}$
  from the inverse funtion theorem,
  one can bound $\norms[\big]{\Dif^{\tau} \Vbfm[\omega]^{-1}}_{\Lipu(D; \Rbbb^d)}$
  independently of $\omega$.
\end{proof}

Now, since for $\Pbbb$-almost any $\omega \in \Omega$ we have
a $C^{\tau,1}$-diffeomorphism from $D \to \Dfrak[\omega]$
we can use the one-to-one correspondence to pull back the model problem
onto the reference domain $D$ instead of considering it
on the actual domain realisations $\Dfrak[\omega]$.
According to the chain rule,
we then have for $v \in H^1(\Dfrak[\omega])$ that
$v \circ \Vbfm[\omega] \in H^1(D)$ and
\begin{equation*}
  (\Grad_\xbfm v) \circ \Vbfm[\omega]
  = \groupp[\big]{\Jbfm[\omega]}^{-\trans} \Grad_\xbfm \groupp[\big]{v \circ \Vbfm[\omega]} .
\end{equation*}

Now, with \eqref{eq:pullbackcoefficientdata}
this leads us to the following formulation
of our model problem \eqref{eq:SPDE2} on the reference domain, cf.\ \cite{HPS16}:
\begin{equation}\label{eq:pullbackswsodp}
  \left\{
  \begin{aligned}
    & \text{Find $\hat{u} \in L_\Pbbb^\infty\groupp[\big]{\Omega; H_0^1(D)}$ such that} \\
    & \qquad \int_{D} \groupa[\big]{\hat{\Abfm}[\omega](\xbfm) \Grad_\xbfm \hat{u}[\omega](\xbfm), \Grad_\xbfm \hat{v}(\xbfm)} \dif\!\xbfm
    = \int_{D} \hat{f}[\omega](\xbfm) \hat{v}(\xbfm) \dif\!\xbfm \\
    & \text{for $\Pbbb$-almost every $\omega \in \Omega$ and all $\hat{v} \in H_0^1(D)$.}
  \end{aligned}
  \right.
\end{equation}
Note, especially, that by the uniformity condition we have that
\begin{equation}\label{eq:Aellipticity}
  \frac{\underline{\sigma}^d}{\overline{\sigma}^2} \leq
  \essinf_{\omega \in \Omega} \essinf_{\xbfm \in D} \lambda_{\min} \groupp[\big]{\hat{\Abfm}[\omega](\xbfm)} \leq
  \esssup_{\omega \in \Omega} \esssup_{\xbfm \in D} \lambda_{\max} \groupp[\big]{\hat{\Abfm}[\omega](\xbfm)} \leq
  \frac{\overline{\sigma}^d}{\underline{\sigma}^2} .
\end{equation}
Without loss of generality,
we assume $\underline{\sigma} \leq 1 \leq \overline{\sigma}$.

From here on,
we assume that the spatial variable $\xbfm$ and the stochastic 
parameter $\omega$ of the random field have been separated by 
the Karhunen\hyp{}Lo\`eve expansion of $\Vbfm$
coming from the mean field $\Mean[\Vbfm]$
and the covariance $\Cov[\Vbfm]$
yielding a parametrised expansion
\begin{equation}
  \label{eq:KLp}
  \Vbfm[\ybfm](\xbfm)
  = \Mean[\Vbfm](\xbfm) + \sum_{k=1}^{\infty} \sigma_k \psibfm_k(\xbfm) y_k ,
\end{equation}
where $\ybfm = (y_k)_{k \in \Nbbb^*} \in \square \isdef \groups{{-1}, 1}^{\Nbbb^*}$ is
a sequence of uncorrelated random variables, see e.g.\ \cite{HPS16};
we denote the pushforward measure of $\Pbbb$ onto $\square$ as $\Pbbb_\ybfm$.
Thus, we then also view all randomness as being parametrised by $\ybfm$,
i.e.\ $\omega$, $\Omega$ and $\Pbbb$ are replaced by $\ybfm$, $\square$ and $\Pbbb_\ybfm$.

We now impose some common assumptions,
which make the Karhunen\hyp{}Lo\`eve expansion computationally feasible.
\begin{assumption}\label{assumption:Vdecay}
  \begin{enumerate}[(1),noitemsep]
  \item The random variables $(y_k)_{k \in \Nbbb^*}$ are independent and identically distributed.
    Moreover, they are uniformly distributed on $\groups[\big]{{-1}, 1}$.
  \item We assume that the $\psibfm_k$ are elements of $C^{\tau,1}(\overline{D}; \Rbbb^d)$
    and that the sequence $\gammabfm = \groupp{\gamma_k}_{k \in \Nbbb}$,
    given by
    \begin{equation*}
      \gamma_k \isdef \norm[\big]{\sigma_k \psibfm_k}_{C^{\tau,1}(\overline{D}; \Rbbb^d)} ,
    \end{equation*}
    is at least in $\ell^1(\Nbbb)$,
    where we have defined $\psibfm_0 \isdef \Mean[\Vbfm]$ and
    $\sigma_0 \isdef 1$.
    Furthermore, we define
    \begin{equation*}
      c_{\gammabfm} = \max\groupb[\big]{\norm{\gammabfm}_{\ell^1(\Nbbb)}, 1} .
    \end{equation*}
  \end{enumerate}
\end{assumption}
%
\section{Regularity}\label{section:regularity}
%
%
To prove the analyticity of the mapping $\hat{u} \colon \square \to H^{\tau+1}(D)$,
we first investigate the analyticity of the mappings
$\hat{\Abfm} \colon \square \to W^{\tau, \infty}(D; \Rbbb^{d \times d}_{\symm})$ and
$\hat{f} \colon \square \to H^{\tau-1}(D)$ in a first subsection.
Based on that analyticity we then prove the analyticity for $\hat{u}$
in the second subsection.

To make the notation less cumbersome,
since we are considering the norm of spaces of the form
$L_{\Pbbb_\ybfm}^\infty\groupp{\square; \Xcal}$,
we introduce the shorthand notation
\begin{equation*}
  \normt{v}_{\Xcal}
  \isdef \norm{v}_{\infty, \square; \Xcal} .
\end{equation*}
We will especially make use it for spaces of the form
$L_{\Pbbb_\ybfm}^\infty\groupp[\big]{\square; W^{\eta, p}(D; \Xcal)}$,
where this then becomes
$\normt{v}_{\eta, p, D; \Xcal} = \norm{v}_{\infty, \square; \eta, p, D; \Xcal}$.
\subsection{Parametric regularity of the diffusion coefficient and the right-hand side}
To provide regularity estimates for the diffusion coefficient $\hat{\Abfm}$
and the right hand side $\hat{f}$,
that are based on the decay of the expansion of $\Vbfm$
as per Assumption~\ref{assumption:Vdecay},
we first note that we can write
\begin{equation}\label{eq:hatAfcomp}
  \hat{\Abfm}[\omega](\xbfm) = \Tbfm\groupp[\big]{\Vbfm[\omega](\xbfm), \Jbfm[\omega](\xbfm)}
  \quad\text{and}\quad
  \hat{f}[\omega](\xbfm) = s\groupp[\big]{\Vbfm[\omega](\xbfm), \Jbfm[\omega](\xbfm)}
\end{equation}
with
\begin{gather}\label{eq:Tdef}
  \Tbfm \colon \Dcal \times \Rbbb^{d \times d}_{\underline{\sigma}, \overline{\sigma}} \to \Rbbb^{d \times d}_{\symm} ,\, (\vbfm, \Mbfm) \mapsto (\Mbfm^\trans \Mbfm)^{-1} \det\Mbfm \\
  \label{eq:sdef}
  s \colon \Dcal \times \Rbbb^{d \times d}_{\underline{\sigma}, \overline{\sigma}} \to \Rbbb ,\, (\vbfm, \Mbfm) \mapsto f(\vbfm) \det\Mbfm ,
\end{gather}
where $\Rbbb^{d \times d}_{\underline{\sigma}, \overline{\sigma}} \isdef \groupb{\Mbfm \in \Rbbb^{d \times d} \,:\, \underline{\sigma} \leq \sigma_{\min}(\Mbfm) \leq \sigma_{\max}(\Mbfm) \leq \overline{\sigma}}$.
Therefore,
we first discuss the regularity of the combined mapping
\begin{equation*}
  (\Vbfm, \Jbfm) \colon \square \to \groupp[\big]{D \to \Dcal \times \Rbbb^{d \times d}_{\underline{\sigma}, \overline{\sigma}}} ,\,
  \ybfm \mapsto \groupp[\Big]{\xbfm \mapsto \groupp[\big]{\Vbfm[\omega](\xbfm), \Jbfm[\omega](\xbfm)}} ,
\end{equation*}
for which we have the following result.
\begin{lemma}\label{lemma:VJybounds}
  We have for all $\alphabfm \in \Nbbb^{\Nbbb^*}_f$ that
  \begin{equation*}
    \normt[\big]{\pdif_\ybfm^\alphabfm (\Vbfm,\Jbfm)}_{\tau, \infty, D}
    \leq k_{\Vbfm\Jbfm} \gammabfm^\alphabfm ,
  \end{equation*}
  where $k_{\Vbfm\Jbfm} \isdef [1 + (\tau+1) d] c_{\tau} c_{\gammabfm}$.
  Here, $c_{\tau}$ denotes the constant coming from the embedding
  $C^{\tau, 1}(\overline{D}; \Rbbb^d) \hookrightarrow W^{\tau+1, \infty}(D; \Rbbb^d)$.
\end{lemma}
\begin{proof}
  By definition we have that $\Jbfm[\ybfm] = \Dif_\xbfm \Vbfm[\ybfm]$
  and so it follows that
  \begin{equation*}
    \Vbfm[\ybfm]
    = \sigma_0 \psibfm_0 + \sum_{k=1}^{\infty} \sigma_k \psibfm_k y_k
    \quad\text{and}\quad
    \Jbfm[\ybfm]
    = \sigma_0 \Dif_\xbfm \psibfm_0 + \sum_{k=1}^{\infty} \sigma_k \Dif_\xbfm \psibfm_k y_k .
  \end{equation*}
  From this we can derive that first order derivatives are given by
  \begin{equation*}
    \pdif_{y_i}\Vbfm[\ybfm]
    = \sigma_i \psibfm_i
    \quad\text{and}\quad
    \pdif_{y_i}\Jbfm[\ybfm]
    = \sigma_i \Dif_\xbfm \psibfm_i
  \end{equation*}
  and all higher derivatives vanish.
  Clearly,
  this affine dependence on $\ybfm$ implies the bounds.
\end{proof}

Next, we supply bounds on the derivatives of the mappings $\Tbfm$ and $s$.
\begin{lemma}\label{lemma:Tbounds}
  The mapping $\Tbfm$ is infinitely Fr\'echet differentiable with
  \begin{equation*}
    \norm[\big]{\Dif^r \Tbfm(\vbfm, \Mbfm)}_{\Bcal^r(\Rbbb^d \times \Rbbb^{d \times d}; \Rbbb^{d \times d}_{\symm})} \leq r! k_\Tbfm c_\Tbfm^r
  \end{equation*}
  for all $(\vbfm, \Mbfm) \in \Dcal \times \Rbbb^{d \times d}_{\underline{\sigma}, \overline{\sigma}}$
  with $k_\Tbfm = \underline{\sigma}^{-2} (2 \overline{\sigma})^d $
  and $c_\Tbfm = 4 (\underline{\sigma}^{-2} \overline{\sigma}^2 + 1)$.
\end{lemma}
\begin{proof}
  We start with the mappings
  \begin{gather*}
    \Tbfm_1 \colon \Dcal \times \Rbbb^{d \times d}_{\underline{\sigma}, \overline{\sigma}} \to \Rbbb^{d \times d}_{\underline{\sigma}, \overline{\sigma}} ,\, (\vbfm, \Mbfm) \mapsto \Mbfm
    \quad\text{and}\quad
    \Tbfm_2 \colon \Dcal \times \Rbbb^{d \times d}_{\underline{\sigma}, \overline{\sigma}} \to \Rbbb^{d \times d}_{\underline{\sigma}, \overline{\sigma}} ,\, (\vbfm, \Mbfm) \mapsto \Mbfm^\trans ,
  \end{gather*}
  which are infinitely Fr\'echet differentiable with
  \begin{equation*}
    \norm[\big]{\Dif^r \Tbfm_i(\vbfm, \Mbfm)}_{\Bcal^r(\Rbbb^d \times \Rbbb^{d \times d}; \Rbbb^{d \times d})} \leq r! k_i c_i^r
  \end{equation*}
  for all $(\vbfm, \Mbfm) \in \Dcal \times \Rbbb^{d \times d}_{\underline{\sigma}, \overline{\sigma}}$,
  $i = 1, 2$ and $k_1 = k_2 = \overline{\sigma}$, $c_1 = c_2 = 1$.
  Then,
  using \cite[Lemma 3]{HS19} we see,
  that the mapping
  \begin{equation*}
    \Tbfm_3 \colon \Dcal \times \Rbbb^{d \times d}_{\underline{\sigma}, \overline{\sigma}} \to \Rbbb^{d \times d}_{\underline{\sigma}^2, \overline{\sigma}^2} ,\, (\vbfm, \Mbfm) \mapsto \Mbfm^\trans \Mbfm
  \end{equation*}
  is infinitely Fr\'echet differentiable with
  \begin{equation*}
    \norm[\big]{\Dif^r \Tbfm_3(\vbfm, \Mbfm)}_{\Bcal^r(\Rbbb^d \times \Rbbb^{d \times d}; \Rbbb^{d \times d})} \leq r! k_3 c_3^r
  \end{equation*}
  for all $(\vbfm, \Mbfm) \in \Dcal \times \Rbbb^{d \times d}_{\underline{\sigma}, \overline{\sigma}}$,
  and $k_3 = \overline{\sigma}^2$, $c_3 = 2$.

  Next,
  we consider the mapping
  \begin{equation*}
    \Minv \colon \Rbbb^{d \times d}_{\underline{\sigma}^2, \overline{\sigma}^2} \to \Rbbb^{d \times d}_{\overline{\sigma}^{-2}, \underline{\sigma}^{-2}} ,\, \Mbfm \mapsto \Mbfm^{-1} .
  \end{equation*}
  Clearly,
  the $r$-th Fr\'echet derivative of $\Minv$ is given by
  \begin{align*}
    \Dif^r \Minv(\Mbfm)\Hbfm_1 \cdots \Hbfm_t
    &= (-1)^r \sum_{\sigma \in S_r} \Mbfm^{-1} \prod_{j=1}^{r} \groupp[\big]{\Hbfm_{\sigma(j)} \Mbfm^{-1}}\\
    &= (-1)^r \sum_{\sigma \in S_r} \Minv(\Mbfm) \prod_{j=1}^{r} \groupp[\big]{\Hbfm_{\sigma(j)} \Minv(\Mbfm)} ,
  \end{align*}
  where $S_r$ is the set of all bijections on the set $\groupb{1, 2, \cdots, r}$.
  Thus,
  we have
  \begin{equation*}
    \norm[\big]{\Dif^r \Minv(\Mbfm)}_{\Bcal^{r}(\Rbbb^{d \times d}; \Rbbb^{d \times d})}
    \leq r! \norm[\big]{\Minv(\Mbfm)}_2^{r+1}
    \leq r! k_{\Minv} c_{\Minv}^r
  \end{equation*}
  for all $\Mbfm \in \Rbbb^{d \times d}_{\underline{\sigma}^2, \overline{\sigma}^2} \to \Rbbb^{d \times d}_{\overline{\sigma}^{-2}, \underline{\sigma}^{-2}}$
  with $k_{\Minv} = c_{\Minv} = \underline{\sigma}^{-2}$.
  Therefore,
  we can use \cite[Lemma 4]{HS19} to see,
  that the mapping
  \begin{equation*}
    \Tbfm_4 \colon \Dcal \times \Rbbb^{d \times d}_{\underline{\sigma}, \overline{\sigma}} \to \Rbbb^{d \times d}_{\overline{\sigma}^{-2}, \underline{\sigma}^{-2}} ,\, (\vbfm, \Mbfm) \mapsto \Minv\groupp[\big]{\Tbfm_3(\vbfm, \Mbfm)} = (\Mbfm^\trans \Mbfm)^{-1}
  \end{equation*}
  is infinitely Fr\'echet differentiable with
  \begin{equation*}
    \norm[\big]{\Dif^r \Tbfm_4(\vbfm, \Mbfm)}_{\Bcal^r(\Rbbb^d \times \Rbbb^{d \times d}; \Rbbb^{d \times d})} \leq r! k_4 c_4^r
  \end{equation*}
  for all $(\vbfm, \Mbfm) \in \Dcal \times \Rbbb^{d \times d}_{\underline{\sigma}, \overline{\sigma}}$,
  and $k_4 = \underline{\sigma}^{-2}$, $c_4 = (\underline{\sigma}^{-2} \overline{\sigma}^2 + 1) 2$.

  Finally,
  we consider the mapping
  \begin{equation*}
    \Det \colon \Rbbb^{d \times d}_{\underline{\sigma}, \overline{\sigma}} \to \Rbbb ,\, \Mbfm \mapsto \Det\Mbfm ,
  \end{equation*}
  which has the $r$-th Fr\'echet derivative of $\Det$ given by
  \begin{equation*}
    \Dif^r \Det(\Mbfm)\Hbfm_1 \cdots \Hbfm_t
    = \sum_{\substack{1 \leq i_1, \ldots, i_r \leq d\\\text{p.w. inequal}}}
    \det\groupp[\big]{\Mbfm_{[i_1, \Hbfm_1], \ldots, [i_r, \Hbfm_r]}} ,
  \end{equation*}
  where $\Mbfm_{[i_1, \Hbfm_1], \ldots, [i_t, \Hbfm_t]}$ denotes the matrix $\Mbfm$
  whose $i_k$-th column is replaced by the $i_k$-th column of the matrix $\Hbfm_k$
  for all $k$ from $1$ to $r$.
  Now, since we can bound the determinant of a matrix by the product of the norms of its columns, i.e.
  \begin{equation*}
    \norms[\Big]{\Det\groupp[\Big]{\begin{bmatrix} \zbfm_1 & \cdots & \zbfm_d\end{bmatrix}}}
    \leq \prod_{j=1}^{d} \norm{\zbfm_j} ,
  \end{equation*}
  and since we know that
  \begin{equation*}
    \norm{\zbfm_j}
    \leq \norm[\Big]{\begin{bmatrix} \zbfm_1 & \cdots & \zbfm_d\end{bmatrix}} .
  \end{equation*}
  it follows that,
  \begin{equation*}
    \norm[\big]{\Dif^r \Det(\Mbfm)}_{\Bcal^{r}(\Rbbb^{d \times d}; \Rbbb)}
    \leq \frac{d!}{(d-r)!} \norm{\Mbfm}^{d-r}
    \leq r! \binom{d}{r} \overline{\sigma}^d
    \leq r! k_{\Det} c_{\Det}^r ,
  \end{equation*}
  with $k_{\Det} = (2 \overline{\sigma})^d$ and $c_{\Det} = 1$.
  As before,
  we can use \cite[Lemma 4]{HS19} to see,
  that the mapping
  \begin{equation*}
    T_5 \colon \Dcal \times \Rbbb^{d \times d}_{\underline{\sigma}, \overline{\sigma}} \to \Rbbb ,\, (\vbfm, \Mbfm) \mapsto \Det\groupp[\big]{\Tbfm_1(\vbfm, \Mbfm)} = \Det\Mbfm
  \end{equation*}
  is infinitely Fr\'echet differentiable with
  \begin{equation*}
    \norm[\big]{\Dif^r T_5(\vbfm, \Mbfm)}_{\Bcal^r(\Rbbb^d \times \Rbbb^{d \times d}; \Rbbb)} \leq r! k_5 c_5^r
  \end{equation*}
  for all $(\vbfm, \Mbfm) \in \Dcal \times \Rbbb^{d \times d}_{\underline{\sigma}, \overline{\sigma}}$,
  and $k_5 = (2 \overline{\sigma})^d$, $c_5 = \overline{\sigma} + 1$.

  Finally,
  the use of \cite[Lemma 4]{HS19} yields the assertion,
  as $\Tbfm(\vbfm, \Mbfm) = T_5(\vbfm, \Mbfm) \Tbfm_4(\vbfm, \Mbfm)$.
\end{proof}

\begin{lemma}\label{lemma:sbounds}
  The mapping $s$ is infinitely Fr\'echet differentiable with
  \begin{equation*}
    \norm[\big]{\Dif^r s(\vbfm, \Mbfm)}_{\Bcal^r(\Rbbb^d \times \Rbbb^{d \times d}; \Rbbb)} \leq r! k_s c_s^r
  \end{equation*}
  for all $(\vbfm, \Mbfm) \in \Dcal \times \Rbbb^{d \times d}_{\underline{\sigma}, \overline{\sigma}}$
  with $k_s = (2 \overline{\sigma})^d k_f$
  and $c_s = 2 \max\groupb[\big]{c_f \displaystyle\max_{\xbfm \in \Dcal} \norm{\xbfm}, \overline{\sigma}} + 2$,
  where $k_f$, $c_f$ are constants such that $\norm[\big]{\Dif^r f(\vbfm)}_{\Bcal^r(\Rbbb^d; \Rbbb)} \leq r! k_f c_f^r$ holds for all $(\vbfm, \Mbfm) \in \Dcal \times \Rbbb^{d \times d}_{\underline{\sigma}, \overline{\sigma}}$.
\end{lemma}
\begin{proof}
  We start with the mapping
  \begin{gather*}
    \sbfm_1 \colon \Dcal \times \Rbbb^{d \times d}_{\underline{\sigma}, \overline{\sigma}} \to \Dcal ,\, (\vbfm, \Mbfm) \mapsto \vbfm ,
  \end{gather*}
  which is infinitely Fr\'echet differentiable with
  \begin{equation*}
    \norm[\big]{\Dif^r \sbfm_1(\vbfm, \Mbfm)}_{\Bcal^r(\Rbbb^d \times \Rbbb^{d \times d}; \Rbbb^{d \times d})} \leq r! k_1 c_1^r
  \end{equation*}
  for all $(\vbfm, \Mbfm) \in \Dcal \times \Rbbb^{d \times d}_{\underline{\sigma}, \overline{\sigma}}$,
  and $k_1 = \displaystyle\max_{\xbfm \in \Dcal} \norm{\xbfm}$, $c_1 = 1$.
  Then,
  using \cite[Lemma 4]{HS19} we see,
  that the mapping
  \begin{equation*}
    s_2 \colon \Dcal \times \Rbbb^{d \times d}_{\underline{\sigma}, \overline{\sigma}} \to \Rbbb ,\, (\vbfm, \Mbfm) \mapsto f\groupp[\big]{\sbfm_1(\vbfm, \Mbfm)} = f(\vbfm)
  \end{equation*}
  is infinitely Fr\'echet differentiable with
  \begin{equation*}
    \norm[\big]{\Dif^r s_2(\vbfm, \Mbfm)}_{\Bcal^r(\Rbbb^d \times \Rbbb^{d \times d}; \Rbbb)} \leq r! k_2 c_2^r
  \end{equation*}
  for all $(\vbfm, \Mbfm) \in \Dcal \times \Rbbb^{d \times d}_{\underline{\sigma}, \overline{\sigma}}$,
  and $k_2 = k_f$, $c_2 = c_f \displaystyle\max_{\xbfm \in \Dcal} \norm{\xbfm} + 1$.

  Moreover,
  as shown in the previous proof we also have that
  \begin{equation*}
    s_3 \colon \Dcal \times \Rbbb^{d \times d}_{\underline{\sigma}, \overline{\sigma}} \to \Rbbb ,\, (\vbfm, \Mbfm) \mapsto \Det\Mbfm
  \end{equation*}
  is infinitely Fr\'echet differentiable with
  \begin{equation*}
    \norm[\big]{\Dif^r s_3(\vbfm, \Mbfm)}_{\Bcal^r(\Rbbb^d \times \Rbbb^{d \times d}; \Rbbb)} \leq r! k_3 c_3^r
  \end{equation*}
  for all $(\vbfm, \Mbfm) \in \Dcal \times \Rbbb^{d \times d}_{\underline{\sigma}, \overline{\sigma}}$,
  and $k_3 = (2 \overline{\sigma})^d$, $c_3 = \overline{\sigma} + 1$.
  Lastly,
  the use of \cite[Lemma 4]{HS19} yields the assertion,
  as $s(\vbfm, \Mbfm) = s_2(\vbfm, \Mbfm) s_3(\vbfm, \Mbfm)$.
\end{proof}

Now,
these results enable us to show the following regularity estimates
for the diffusion coefficient $\hat{\Abfm}$ and the right hand side $\hat{f}$.
\begin{theorem}\label{theorem:Afybounds}
  We know for all $\alphabfm \in \Nbbb^{\Nbbb^*}_f$ that
  \begin{equation*}
    \normt[\big]{\pdif_\ybfm^\alphabfm \hat{\Abfm}}_{\tau, \infty, D; \Rbbb_\symm^{d \times d}}
    \leq \norms{\alphabfm}! k_{\hat{\Abfm}} c_{\hat{\Abfm}}^{\norms{\alphabfm}} \gammabfm^{\alphabfm}
    \quad\text{and}\quad
    \normt[\big]{\pdif_\ybfm^\alphabfm \hat{f}}_{\tau-1, 2, D; \Rbbb}
    \leq \norms{\alphabfm}! k_{\hat{f}} c_{\hat{f}}^{\norms{\alphabfm}} \gammabfm^{\alphabfm} ,
  \end{equation*}
  where
  \begin{equation*}
    k_{\hat{\Abfm}} \isdef k_\Tbfm \sum_{r=0}^{\tau} c_\Tbfm^r k_{\Vbfm\Jbfm}^r
    \, , \quad
    c_{\hat{\Abfm}} \isdef 2 c_\Tbfm k_{\Vbfm\Jbfm} + 1
    \, , \quad
    k_{\hat{f}} \isdef \sqrt{\frac{\norms{\Dcal}}{\underline{\sigma}^{d}}} k_s \sum_{r=0}^{\tau} c_s^r k_{\Vbfm\Jbfm}^r
    \quad\text{and}\quad
    c_{\hat{f}} \isdef 2 c_s k_{\Vbfm\Jbfm} + 1 .
  \end{equation*}
\end{theorem}
\begin{proof}
  Because $\hat{\Abfm} = \Tbfm \circ (\Vbfm, \Jbfm)$,
  we can employ \cite[Lemma 8]{HS19}
  to arrive at
  \begin{align*}
    \normt{\hat{\Abfm}}_{\tau, \infty, D; \Rbbb_\symm^{d \times d}}
    &\leq \sum_{r=0}^{\tau} \frac{1}{r!}
    \normt[\big]{\Dif^{r} \Tbfm \circ (\Vbfm, \Jbfm)}_{\infty, D; \Bcal^{r}(\Rbbb^d \times \Rbbb^{d \times d}; \Rbbb_\symm^{d \times d}))}
    \normt{(\Vbfm, \Jbfm)}_{\eta, \tau, D; \Rbbb^d \times \Rbbb^{d \times d}}^r \\
    &\leq k_\Tbfm \sum_{r=0}^{\tau} c_\Tbfm^r k_{\Vbfm\Jbfm}^r
    \leq k_{\hat{\Abfm}}
  \end{align*}
  as well as, for $\alphabfm \neq \zerobfm$,
  \begin{align*}
    &\normt[\big]{\pdif_\ybfm^\alphabfm \hat{\Abfm}}_{\tau, \infty, D; \Rbbb_\symm^{d \times d}} \\
    &\quad \leq \alphabfm!
    \sum_{s=1}^{\norms{\alphabfm}} \frac{1}{s!}
    \groupp[\bigg]{
      \sum_{r=0}^{\tau} \frac{1}{r!}
      \normt[\big]{\Dif^{r+s} \Tbfm \circ (\Vbfm, \Jbfm)}_{\infty, D; \Bcal^{r+s}(\Rbbb^d \times \Rbbb^{d \times d}; \Rbbb_\symm^{d \times d})}
      \normt{(\Vbfm, \Jbfm)}_{\tau, \infty, D; \Rbbb^d \times \Rbbb^{d \times d}}^r} \\
    &\quad \qquad \qquad \sum_{C(\alphabfm, s)}
    \prod_{j=1}^s \frac{1}{\betabfm_j!}
    \normt[\big]{\pdif_\ybfm^{\betabfm_j} (\Vbfm, \Jbfm)}_{\tau, \infty, D; \Rbbb^d \times \Rbbb^{d \times d}} \\
    &\quad \leq \alphabfm!
    \sum_{s=1}^{\norms{\alphabfm}} \frac{1}{s!}
    \groupp[\bigg]{
      \sum_{r=0}^{\tau} \frac{1}{r!}
      (r+s)! k_\Tbfm c_\Tbfm^{r+s} k_{\Vbfm\Jbfm}^r}
    \sum_{C(\alphabfm, s)}
    \prod_{j=1}^s \frac{1}{\betabfm_j!}
    k_{\Vbfm\Jbfm} \gammabfm^{\betabfm_j} \\
    &\quad \leq \gammabfm^{\alphabfm}
    k_\Tbfm
    \groupp[\bigg]{
      \sum_{r=0}^{\tau} 2^r c_\Tbfm^r k_{\Vbfm\Jbfm}^r}
    \sum_{s=1}^{\norms{\alphabfm}}
    2^s c_\Tbfm^s k_{\Vbfm\Jbfm}^s
    \alphabfm!
    \sum_{C(\alphabfm, s)}
    \prod_{j=1}^s \frac{1}{\betabfm_j!} \\
    &\quad \leq \gammabfm^{\alphabfm} \norms{\alphabfm}!
    k_\Tbfm
    \groupp[\bigg]{
      \sum_{r=0}^{\tau} 2^r c_\Tbfm^r k_{\Vbfm\Jbfm}^r}
    \sum_{s=1}^{\norms{\alphabfm}}
    2^s c_\Tbfm^s k_{\Vbfm\Jbfm}^s
    \binom{\norms{\alphabfm} - 1}{s - 1} \\
    &\quad \leq \gammabfm^{\alphabfm} \norms{\alphabfm}!
    k_\Tbfm
    \groupp[\bigg]{
      \sum_{r=0}^{\tau} 2^r c_\Tbfm^r k_{\Vbfm\Jbfm}^r}
    (2 c_\Tbfm k_{\Vbfm\Jbfm} + 1)^{\norms{\alphabfm}} ,
  \end{align*}
  where $C(\alphabfm, s)$ is the set of all compositions of the multi\hyp{}index $\alphabfm$
  into $s$ non\hyp{}vanishing multi\hyp{}indices $\betabfm_1, \ldots, \betabfm_s$,
  see \cite{HS19},
  and we make use of the combinatorial identity
  \begin{equation*}
    \alphabfm!
    \sum_{C(\alphabfm, s)}
    \prod_{j=1}^s \frac{\norms{\betabfm_j}!}{\betabfm_j!}
    = \norms{\alphabfm}! \binom{\norms{\alphabfm} - 1}{s - 1} .
  \end{equation*}
  This proves the assertion for $\hat{\Abfm}$,
  while the assertion for $\hat{f}$ follows analogously
  after remarking that
  \begin{align*}
    \normt[\big]{\Dif^{t} s \circ (\Vbfm, \Jbfm)}_{2, D; \Bcal^{t}(\Rbbb^d \times \Rbbb^{d \times d}; \Rbbb)}
    &= \esssup_{\ybfm \in \square}
    \norm[\big]{\Dif^t s \circ (\Vbfm[\ybfm], \Jbfm[\ybfm])}_{2, D; \Bcal^{t}(\Rbbb^d \times \Rbbb^{d \times d}; \Rbbb)} \\
    &\leq \esssup_{\ybfm \in \square}
    \norm[\big]{\Dif^t s}_{2, \Dfrak[\ybfm]; \Bcal^{t}(\Rbbb^d \times \Rbbb^{d \times d}; \Rbbb)}
    \sqrt{\underline{\sigma}^{-d}} \\
    &\leq \esssup_{\ybfm \in \square}
    \norm[\big]{\Dif^t s}_{\infty, \Dfrak[\ybfm]; \Bcal^{t}(\Rbbb^d \times \Rbbb^{d \times d}; \Rbbb)}
    \sqrt{\norms{\Dfrak[\ybfm]} \underline{\sigma}^{-d}} \\
    &\leq 
    \norm[\big]{\Dif^t s}_{\infty, \Dcal; \Bcal^{t}(\Rbbb^d \times \Rbbb^{d \times d}; \Rbbb)}
    \sqrt{\norms{\Dcal} \underline{\sigma}^{-d}} \\
    &\leq t! \sqrt{\norms{\Dcal}\underline{\sigma}^{-d}} k_s c_s^t . \qedhere
  \end{align*}
\end{proof}
\subsection{Parametric regularity of the solution}
For this subsection, we require an elliptic regularity result,
which we state as an assumption:

\begin{assumption}\label{assumption:er}
  Let $D$ be a sufficiently smooth such that,
  for all
  \begin{equation*}
    \Bbfm \in C^{\tau-1, 1}(\overline{D}; \Rbbb_{\symm}^{d \times d})
  \end{equation*}
  that fulfil \eqref{eq:Aellipticity},
  we have that
  the problem of solving
  \begin{equation*}
    \groupp[\Big]{\Bbfm \Grad_{\xbfm} u, \Grad_x v}_{L^2(D; \Rbbb^d)} = \groupp{h, v}_{L^2(D)}
  \end{equation*}
  for any $h \in H^{\tau-1}(D)$ has a unique solution $u \in H_0^1(D)$,
  which also lies in $H^{\tau+1}(D)$,
  with
  \begin{equation*}
    \norm{u}_{\tau+1, 2, D}
    \leq C_{er} \norm{h}_{\tau-1, 2, D} ,
  \end{equation*}
  where $C_{er}$ only depends on $D$, $\underline{\sigma}$, $\overline{\sigma}$, $\tau$
  and a bound on $\norm{\Bbfm}_{C^{\tau-1, 1}(\overline{D}; \Rbbb_{\symm}^{d \times d})}$.
\end{assumption}
Such an elliptic regularity estimate for example is known for $\tau = 1$,
when the domain is convex and bounded,
see \cite[Propositions 3.2.1.2 and 3.1.3.1]{Grisvard}.
The elliptic regularity estimate is also known to hold for $\tau \geq 1$ and $d = 2$,
when the domain's boundary is smooth,
see \cite{BLN17}.

This obviously directly implies the following result.
\begin{lemma}\label{lemma:eubound}
  The unique solution $\hat{u} \in L_{\Pbbb_\ybfm}^\infty\groupp[\big]{\square; H_0^1(D)}$
  of \eqref{eq:pullbackswsodp} fulfils
  $\hat{u} \in L_{\Pbbb_\ybfm}^\infty\groupp[\big]{\square; H^{\tau+1}(D)}$,
  with
  \begin{equation*}
    \normt{\hat{u}}_{\tau+1, 2, D}
    \leq C_{er} \normt{f}_{\tau-1, 2, D} .
  \end{equation*}
\end{lemma}

Moreover, this higher spatial regularity also carries over to the derivates $\pdif_\ybfm^\alphabfm \hat{u}$.
\begin{theorem}\label{theorem:euybounds}
  For almost every $\ybfm \in \square$,
  the derivatives of the solution $\hat{u}$ of \eqref{eq:pullbackswsodp} satisfy
  \begin{align*}
    \normt[\big]{\pdif_\ybfm^\alphabfm \hat{u}}_{k+1, 2, D}
    &\leq \norms{\alphabfm}!
    \max\groupb[\big]{2, 3 C_{er} \tau^2 d^2 k_{\hat{\Abfm}}, 3 C_{er} k_{\hat{f}}} \\
    &\qquad \groupp[\Big]{
      \max\groupb[\big]{2, 3 C_{er} \tau^2 d^2 k_{\hat{\Abfm}}, 3 C_{er} k_{\hat{f}}}
      \max\groupb[\big]{c_{\hat{f}}, c_{\hat{\Abfm}}}}^{\norms{\alphabfm}}
    \gammabfm^{\alphabfm} .
  \end{align*}
\end{theorem}
\begin{proof}
  By differentiation of the variational formulation \eqref{eq:pullbackswsodp}
  with respect to $\ybfm$ we arrive,
  for arbitrary $v \in H_0^1(D)$, at
  \begin{equation*}
    \groupp[\Big]{\pdif_\ybfm^\alphabfm \groupp[\big]{\hat{\Abfm} \Grad_\xbfm \hat{u}},
      \Grad_\xbfm \hat{v}}_{L^2(D; \Rbbb^d)} = \groupp[\Big]{\pdif_\ybfm^\alphabfm\hat{f}, \hat{v}}_{L^2(D; \Rbbb)} .
  \end{equation*}
  Applying the Leibniz rule on the left\hyp{}hand side yields
  \begin{equation*}
    \groupp[\bigg]{\sum_{\betabfm \leq \alphabfm} \binom{\alphabfm}{\betabfm}
    \pdif_\ybfm^{\alphabfm-\betabfm} \hat{\Abfm}
    \pdif_\ybfm^{\betabfm} \Grad_\xbfm \hat{u} ,
    \Grad_\xbfm \hat{v}}_{L^2(D; \Rbbb^d)} = \groupp[\Big]{\pdif_\ybfm^\alphabfm\hat{f}, \hat{v}}_{L^2(D; \Rbbb)} .
  \end{equation*}
  Then, by rearranging and using the linearity of the gradient, we find
  \begin{equation*}
    \groupp[\Big]{
      \hat{\Abfm}
      \Grad_\xbfm \pdif_\ybfm^{\alphabfm} \hat{u} ,
      \Grad_\xbfm \hat{v}}_{L^2(D; \Rbbb^d)}
    = \groupp[\bigg]{
      \sum_{\betabfm < \alphabfm} \binom{\alphabfm}{\betabfm}
      \pdif_\ybfm^{\alphabfm-\betabfm} \hat{\Abfm}
      \Grad_\xbfm \pdif_\ybfm^{\betabfm} \hat{u} ,
      \Grad_\xbfm \hat{v}}_{L^2(D; \Rbbb^d)}
    + \groupp[\Big]{\pdif_\ybfm^\alphabfm\hat{f}, \hat{v}}_{L^2(D; \Rbbb)} .
  \end{equation*}
  Using Green's identity, we can then write
  \begin{equation*}
    \groupp[\Big]{
      \hat{\Abfm}
      \Grad_\xbfm \pdif_\ybfm^{\alphabfm} \hat{u} ,
      \Grad_\xbfm \hat{v}}_{L^2(D; \Rbbb^d)}
    = \groupp[\bigg]{
      \sum_{\betabfm < \alphabfm} \binom{\alphabfm}{\betabfm}
      \Div_\xbfm \groupp[\Big]{
        \pdif_\ybfm^{\alphabfm-\betabfm} \hat{\Abfm}
        \Grad_\xbfm \pdif_\ybfm^{\betabfm} \hat{u}}
      + \pdif_\ybfm^\alphabfm\hat{f},
      \hat{v}}_{L^2(D; \Rbbb)} .
  \end{equation*}

  Thus, we arrive at
  \begin{align*}
    \normt[\big]{\pdif_\ybfm^{\alphabfm} \hat{u}}_{\tau+1, 2, D}
    &\leq C_{er}
    \normt[\bigg]{\pdif_\ybfm^\alphabfm\hat{f}
      + \sum_{\betabfm < \alphabfm} \binom{\alphabfm}{\betabfm}
      \Div_\xbfm \groupp[\Big]{
        \pdif_\ybfm^{\alphabfm-\betabfm} \hat{\Abfm}
        \Grad_\xbfm \pdif_\ybfm^{\betabfm} \hat{u}}}_{\tau-1, 2, D} \\
    &\leq C_{er} \groupp[\Bigg]{
      \normt[\big]{\pdif_\ybfm^\alphabfm\hat{f}}_{\tau-1, 2, D}
      + \sum_{\betabfm < \alphabfm} \binom{\alphabfm}{\betabfm}
      \normt[\Big]{\Div_\xbfm
        \groupp[\Big]{
          \pdif_\ybfm^{\alphabfm-\betabfm} \hat{\Abfm}
          \Grad_\xbfm \pdif_\ybfm^{\betabfm} \hat{u}}}_{\tau-1, 2, D}} \\
    &\leq C_{er} \groupp[\Bigg]{
      \normt[\big]{\pdif_\ybfm^\alphabfm\hat{f}}_{\tau-1, 2, D}
      + \sum_{\betabfm < \alphabfm} \binom{\alphabfm}{\betabfm}
      \tau d \normt[\Big]{
        \pdif_\ybfm^{\alphabfm-\betabfm} \hat{\Abfm}
        \Grad_\xbfm \pdif_\ybfm^{\betabfm} \hat{u}}_{\tau, 2, D}} \\
    &\leq C_{er} \groupp[\Bigg]{
      \normt[\big]{\pdif_\ybfm^\alphabfm\hat{f}}_{\tau-1, 2, D}
      + \sum_{\betabfm < \alphabfm} \binom{\alphabfm}{\betabfm}
      \tau d
      \normt[\big]{\pdif_\ybfm^{\alphabfm-\betabfm} \hat{\Abfm}}_{\tau, \infty, D}
      \normt[\big]{\Grad_\xbfm \pdif_\ybfm^{\betabfm} \hat{u}}_{\tau, 2, D}} \\
    &\leq C_{er} \groupp[\Bigg]{
      \norms{\alphabfm}!
      k_{\hat{f}}
      c_{\hat{f}}^{\norms{\alphabfm}}
      \gammabfm^{\alphabfm}
      + \sum_{\betabfm < \alphabfm} \binom{\alphabfm}{\betabfm}
      \norms{\alphabfm-\betabfm}!
      \tau^2 d^2 k_{\hat{\Abfm}}
      c_{\hat{\Abfm}}^{\norms{\alphabfm-\betabfm}}
      \gammabfm^{\alphabfm-\betabfm}
      \normt[\big]{\pdif_\ybfm^{\betabfm} \hat{u}}_{\tau+1, 2, D}}
  \end{align*}
  from which we derive
  \begin{equation*}
    \normt[\big]{\pdif_\ybfm^\alphabfm \hat{u}}_{\tau+1, 2, D}
    \leq \frac{k}{3}
    \norms{\alphabfm}!
     c^{\norms{\alphabfm}}
     \gammabfm^{\alphabfm}
     + \frac{k}{3}
     \sum_{\betabfm < \alphabfm} \binom{\alphabfm}{\betabfm}
     \norms{\alphabfm-\betabfm}!
     c^{\norms{\alphabfm-\betabfm}}
     \gammabfm^{\alphabfm-\betabfm}
     \normt[\big]{\pdif_\ybfm^{\betabfm} \hat{u}}_{\tau+1, 2, D} ,
  \end{equation*}
  where
  \begin{equation*}
    k \isdef
    \max\groupb[\Big]{2, 3 C_{er} \tau^2 d^2 k_{\hat{\Abfm}}, 3 C_{er} k_{\hat{f}}}
  \end{equation*}
  and
  \begin{equation*}
    c \isdef \max\groupb[\Big]{c_{\hat{f}}, c_{\hat{\Abfm}}}
  \end{equation*}

  We note that, by definition of $k$, we have $k \geq 2$
  and furthermore, because of Lemma~\ref{lemma:eubound}, we also have that
  $\normt[\big]{\hat{u}}_{\tau+1, 2, D} \leq C_{er} k_{\hat{f}} \leq k$,
  which means that the assertion is true for $\norms{\alphabfm} = 0$.
  Thus, we can use an induction over $\norms{\alphabfm}$ to prove the hypothesis
  \begin{equation*}
    \norm[\big]{\pdif_\ybfm^\alphabfm \hat{u}}_{\tau+1, 2, D} 
    \leq \norms{\alphabfm}! \gammabfm^\alphabfm k (k c)^{\norms{\alphabfm}}
  \end{equation*}
  for $\norms{\alphabfm} > 0$.
  
  Let the assertions hold for all $\alphabfm$,
  which satisfy $\norms{\alphabfm} \leq n-1$ for some $n \geq 1$.
  Then, we know for all $\alphabfm$ with $\norms{\alphabfm} = n$ that
  \begin{align*}
    \normt[\big]{\pdif_\ybfm^\alphabfm \hat{u}}_{\tau+1, 2, D}
    &\leq \frac{k}{3}
    \norms{\alphabfm}!
    c^{\norms{\alphabfm}}
    \gammabfm^{\alphabfm}
    + \frac{k}{3}
    \sum_{\betabfm < \alphabfm} \binom{\alphabfm}{\betabfm}
    \norms{\alphabfm-\betabfm}!
    c^{\norms{\alphabfm-\betabfm}}
    \gammabfm^{\alphabfm-\betabfm}
    \normt[\big]{\pdif_\ybfm^{\betabfm} \hat{u}}_{\tau+1, 2, D} \\
    &\leq \frac{k}{3}
    \norms{\alphabfm}!
    c^{\norms{\alphabfm}}
    \gammabfm^{\alphabfm}
    + \frac{k}{3}
    \sum_{\betabfm < \alphabfm} \binom{\alphabfm}{\betabfm}
    \norms{\alphabfm-\betabfm}!
    c^{\norms{\alphabfm-\betabfm}}
    \gammabfm^{\alphabfm-\betabfm}
    \norms{\betabfm}! \gammabfm^\betabfm k (k c)^{\norms{\betabfm}} \\
    &\leq \frac{k}{3}
    \norms{\alphabfm}!
    c^{\norms{\alphabfm}}
    \gammabfm^{\alphabfm}
    + \frac{k}{3}
    \gammabfm^{\alphabfm}
    c^{\norms{\alphabfm}}
    k \sum_{j=0}^{n-1}
    k^{j}
    \sum_{\substack{\betabfm < \alphabfm\\\norms{\betabfm}=j}}
    \binom{\alphabfm}{\betabfm}
    \norms{\alphabfm-\betabfm}! \norms{\betabfm}! .
  \end{align*}
  Making use of the combinatorial identity
  \begin{equation*}
    \sum_{\substack{\betabfm < \alphabfm\\\norms{\betabfm}=j}}
    \binom{\alphabfm}{\betabfm}
    \norms{\alphabfm-\betabfm}! \norms{\betabfm}!
    = \norms{\alphabfm}! ,
  \end{equation*}
  see \cite{HS19},
  yields
  \begin{equation*}
    k \sum_{j=0}^{n-1}
    k^{j}
    \sum_{\substack{\betabfm < \alphabfm\\\norms{\betabfm}=j}}
    \binom{\alphabfm}{\betabfm}
    \norms{\alphabfm-\betabfm}! \norms{\betabfm}!
    = \norms{\alphabfm}!
    k \sum_{j=0}^{n-1} k^{j}
    = \norms{\alphabfm}!
    k \frac{k^{\norms{\alphabfm}}}{k - 1}
    \leq \norms{\alphabfm}! 2 k^{\norms{\alphabfm}} ,
  \end{equation*}
  as $k \geq 2$ implies that $2 (k-1) \geq k$.
  Finally, we arrive at
  \begin{equation*}
    \normt[\big]{\pdif_\ybfm^\alphabfm \hat{u}}_{\tau+1 ,2}
    \leq \frac{k}{3}
    \norms{\alphabfm}!
    c^{\norms{\alphabfm}}
    \gammabfm^{\alphabfm}
    + \frac{k}{3}
    \gammabfm^{\alphabfm}
    c^{\norms{\alphabfm}}
    \norms{\alphabfm}!
    2 k^{\norms{\alphabfm}}
    \leq
    \norms{\alphabfm}!
    k
    (k c)^{\norms{\alphabfm}}
    \gammabfm^{\alphabfm}
  \end{equation*}
  which completes the proof.
\end{proof}
%
\section{The coupling of FEM and BEM}\label{section:coupling}
While we have considered general random domain mappings in the previous subsections,
we will now restrict them according to the remarks made in the introduction.
That is, we assume for the rest of the article that we are given
a random boundary description, $\Gamma[\omega]$, and a fixed, deterministic subdomain $B$,
which describe our random domain, compare Figure~\ref{fig:domain} when $\Gamma = \Gamma[\omega]$.
Moreover, we assume that we are interested the some quantity of interest
that is based on the knowledge of $u|_{B}$ as in \eqref{eq:QoI}.

We will assume that there is a random domain mapping $\Vbfm$ which fulfils
the Assumption~\ref{assumption:Vdecay} as well as fulfilling
$\Vbfm[\omega]|_B = \Id_B$ and $\Vbfm[\omega](\partial D) = \Gamma[\omega]$
for almost any $\omega$.
Then,
we know from the previous section that $\hat{u} \colon \square \to H^{\tau+1}(D)$
is analytic which also implies that $u|_B \colon \square \to H^{\tau+1}(B)$
is analytic.

So, to be able to use multilevel quadrature to compute the quantity of interest efficiently,
we consider a formulation here,
that enables us to compute the Galerkin solution $u_{h}[\omega] \in H^1(B)$
with a mesh on $B$ but without needing a mesh on $\Dfrak[\omega] \setminus B$
or needing the knowledge of the random domain mapping.
One arrives at such a formulation by reformulate the boundary value 
problem as two coupled problems involving only boundary 
integral equations on the random boundary $\Gamma[\omega]$.
\subsection{Newton potential}
For sake of simplicity in representation, we shall
restrict ourselves the deterministic boundary value problem
\begin{equation}\label{eq:PDE0}
  -\Laplace u = f\ \text{in $D$}, \quad
  	u = 0\ \text{on $\Gamma\isdef\partial D$},
\end{equation}
i.e., the domain $D$ is assumed to be fixed.
Of course, when applying a sampling method 
for \eqref{eq:SPDE1}, the underlying domains are 
always different. In order to 
resolve the inhomogeneity in \eqref{eq:PDE0}, we 
introduce a Newton potential $\Ncal_f$ which
satisfies
\begin{equation}	\label{eq:newton}
  -\Laplace \Ncal_f = f \quad \text{in $\widetilde{D}$}.
\end{equation}
Here, $\widetilde{D}$ is a sufficiently large domain
containing $\Dfrak[\omega]$ almost surely. 

The Newton potential is supposed to be explicitly known 
like in our numerical example (see Section~\ref{section:results})
or computed with sufficiently high accuracy. Especially, 
since the domain $\widetilde{D}$ can be chosen fairly 
simple, one can apply finite elements based on tensor 
products of higher order spline functions (in $[-R,R]^d$) or 
dual reciprocity methods. Notice that the Newton potential 
has to be computed only once in advance.

By making the ansatz
\begin{equation}	\label{eq:ansatz}
  u = \Ncal_f + \tilde{u}
\end{equation}
and setting $\tilde{g} \isdef g-\Ncal_f$, we arrive at 
the problem of seeking a harmonic function $\tilde{u}$ 
which solves the following Dirichlet problem for the 
Laplacian
\begin{equation}	\label{eq:PDE1}
   \Laplace  \tilde{u} = 0\ \text{in $D$},\quad
	\tilde{u} = \tilde{g}\ \text{on $\Gamma$}.
\end{equation}
Now, we are able to apply the coupling of 
finite elements and boundary elements.

\subsection{Reformulation as a coupled problem}
For the subdomain $B\subset D$, we set $\Sigma 
\isdef \partial B$, see Figure~\ref{fig:domain} for an illustration. 
The normal vectors $\nbfm$ at $\Gamma$ and $\Sigma$ 
are assumed to point into $D\setminus\overline{B}$. We 
shall split \eqref{eq:PDE1} in two coupled boundary 
value problems in accordance with
\begin{alignat}{3}
  \Laplace\tilde{u} &= 0  		  && \qquad \text{in}\ B, \nonumber\\
   \Laplace\tilde{u} &= 0  		  && \qquad \text{in}\ D\setminus\overline{B}, \nonumber\\
 \lim_{\begin{smallmatrix}\zbfm\to\xbfm\\\zbfm\in B\end{smallmatrix}}\tilde{u}(\zbfm)
  	&= \lim_{\begin{smallmatrix}\zbfm\to\xbfm\\\zbfm\in D\setminus\overline{B}\end{smallmatrix}}\tilde{u}(\zbfm)
				  && \qquad \text{for all}\ \xbfm\in\Sigma,\label{eq:couple}\\
 \lim_{\begin{smallmatrix}\zbfm\to\xbfm\\\zbfm\in B\end{smallmatrix}}\frac{\partial\tilde{u}}{\partial \nbfm}(\zbfm)
  	&= \lim_{\begin{smallmatrix}\zbfm\to\xbfm\\\zbfm\in D\setminus\overline{B}\end{smallmatrix}}
		\frac{\partial\tilde{u}}{\partial \nbfm}(\zbfm)
				  && \qquad \text{for all}\ \xbfm\in\Sigma,\nonumber\\
  \tilde{u} &= \tilde{g} &&\qquad \text{on}\ \Gamma.\nonumber
\end{alignat}

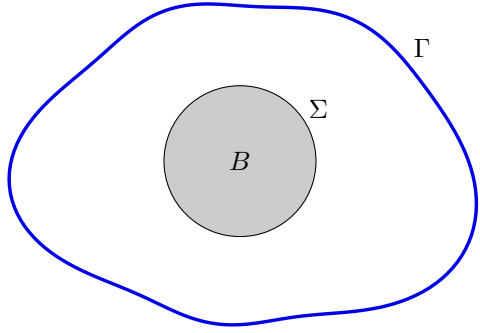
\begin{figure}
  \begin{center}
    \begin{tikzpicture}
      \draw[draw=black, fill=black!20] (0cm, 0cm) ellipse (1cm);
      \node (0,0) {\includegraphics[width=6.78cm]{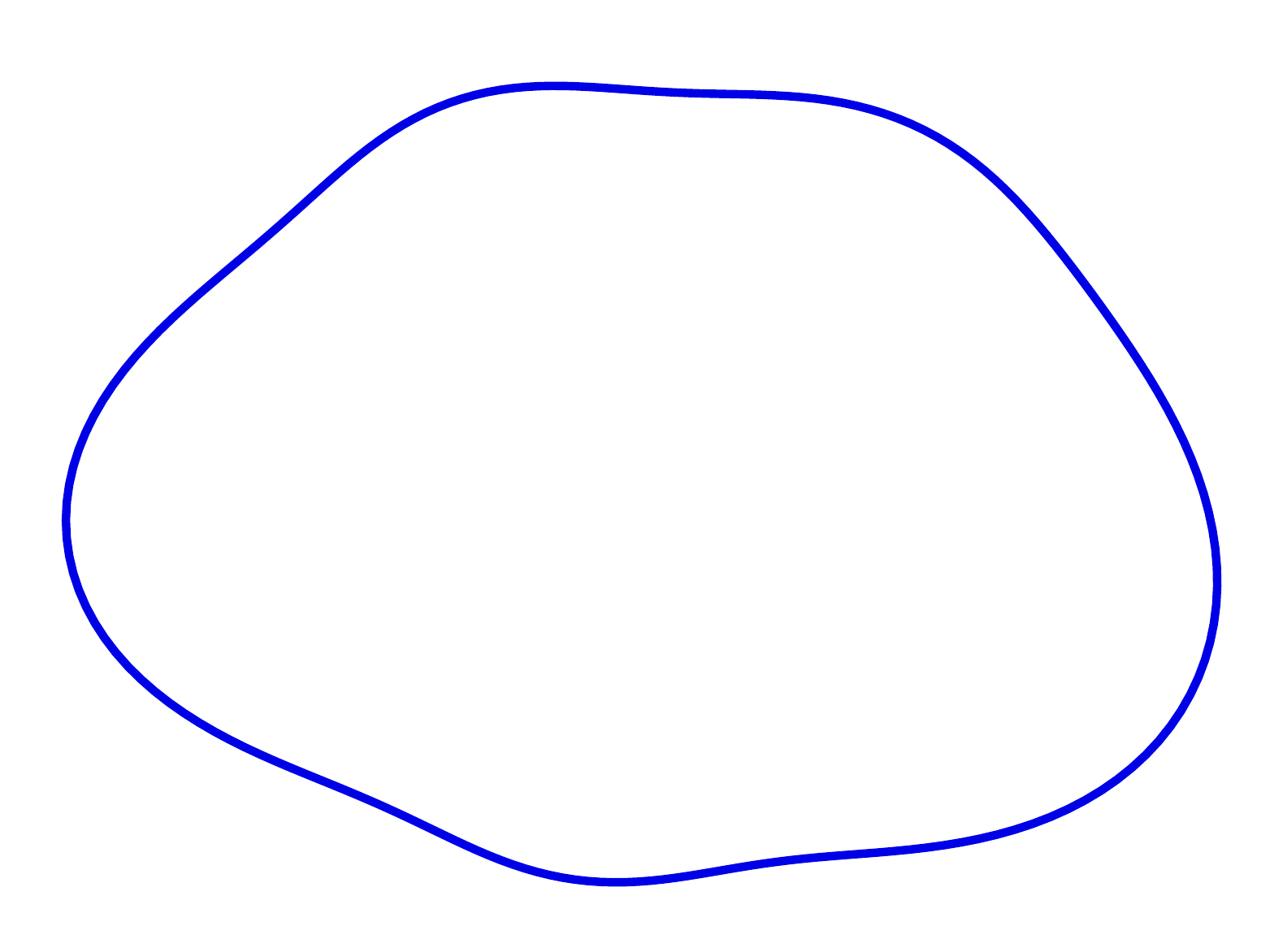}};
      
      \draw (0cm, 0cm) node {$B$};
      \draw (30:0.9cm) node[anchor=south west] {$\Sigma$};
      \draw (30:2.5cm) node[anchor=south west] {$\Gamma$};
    \end{tikzpicture}
    \caption{\label{fig:domain}
      The domain $D$, the subdomain $B$,
      and the boundaries $\Gamma = \partial D$
      and $\Sigma = \partial B$.}
  \end{center}
\end{figure}

In order to derive suitable boundary integral equations 
for the problem in $D\setminus\overline{B}$, we define the 
single layer operator $\Vcal_{\Phi\Psi}$, the double 
layer operator $\Kcal_{\Phi\Psi}$ and its adjoint 
$\Kcal_{\Psi\Phi}^\star$, and the hypersingular 
operator $\Wcal_{\Phi\Psi}$ with respect to 
the boundaries $\Phi,\Psi\in\{\Gamma,\Sigma\}$ by
\[
  \left.\begin{array}{l}
  \displaystyle{(\Vcal_{\Phi\Psi}v)(\xbfm)
  	\isdef \int_\Phi G(\xbfm,\zbfm)v(\zbfm)\dif\!\sigma_\zbfm,}\\
  \displaystyle{(\Kcal_{\Phi\Psi}v)(\xbfm) \isdef \int_\Phi
  	\frac{\partial G(\xbfm,\zbfm)}{\partial\nbfm(\zbfm)} v(\zbfm)\dif\!\sigma_\zbfm},\\
  \displaystyle{(\Kcal_{\Phi\Psi}^\star v)(\xbfm) \isdef \int_\Phi
  	\frac{\partial}{\partial\nbfm(\xbfm)} G(\xbfm,\zbfm)v(\zbfm)\dif\!\sigma_\zbfm},\\
  \displaystyle{(\Wcal_{\Phi\Psi}v)(\xbfm) \isdef - \frac{\partial}{\partial\nbfm(\xbfm)} \int_\Phi
  	\frac{\partial G(\xbfm,\zbfm)}{\partial \nbfm(\zbfm)}v(\zbfm)\dif\!\sigma_\zbfm},
		\end{array}\quad\right\}\quad\xbfm\in\Psi.
\]
Here, $G(\xbfm,\zbfm)$ denotes the fundamental 
solution of the Laplacian which is given by
\[
  G(\xbfm,\zbfm) = \begin{cases} -\frac{1}{2\pi}\log\|\xbfm-\zbfm\|,&d=2,\\
  	\frac{1}{4\pi\|\xbfm-\zbfm\|},&d=3.\end{cases}
\]

By introducing the variables $\sigma_\Sigma \isdef ({\partial\tilde{u}}
/{\partial \nbfm})|_\Sigma$ and $\sigma_\Gamma \isdef ({\partial\tilde{u}}
/{\partial \nbfm})|_\Gamma$, the coupled system \eqref{eq:couple} 
yields the following nonlocal boundary value problem: Find 
$(\tilde{u},\sigma_\Sigma,\sigma_\Gamma)$ such that
\begin{alignat}{3}
  \Laplace\tilde{u} &= 0 && \qquad \text{in}\ B,    \nonumber \\
  \Laplace\tilde{u} &= 0	&& \qquad \text{on}\ \Omega\setminus\overline{B},  \nonumber \\
  -\Wcal_{\Sigma\Sigma}\tilde{u} - \Wcal_{\Gamma\Sigma}\tilde{g}
  	+ \Big(\frac{1}{2}-\Kcal_{\Sigma\Sigma}^\star\Big)\sigma_\Sigma
	  - \Kcal_{\Gamma\Sigma}^\star\sigma_\Gamma
		&= \sigma_\Sigma  && \qquad \text{on}\ \Sigma,
  	\label{continuous problem}\\
  \Big(\frac{1}{2}-\Kcal_{\Sigma\Sigma}\Big)\tilde{u}-\Kcal_{\Gamma\Sigma}\tilde{g}
  	+ \Vcal_{\Sigma\Sigma}\sigma_\Sigma + \Vcal_{\Gamma\Sigma}\sigma_\Gamma &= 0
		&& \qquad \text{on}\ \Sigma,\nonumber\\
  -\Kcal_{\Sigma\Gamma}\tilde{u}+\Big(\frac{1}{2}-\Kcal_{\Gamma\Gamma}\Big)\tilde{g}
  	+ \Vcal_{\Sigma\Gamma}\sigma_\Sigma + \Vcal_{\Gamma\Gamma}\sigma_\Gamma &= 0
		&& \qquad \text{on}\ \Gamma.\nonumber
\end{alignat}
This system is the so-called {\em two integral formulation}, which
is equivalent to our original model problem \eqref{eq:PDE1},
see for example \cite{CS88,H90}.

\subsection{Variational formulation}
We next introduce the product space $\Hcal \isdef 
H^1(B)\times H^{-1/2}(\Sigma)\times H^{-1/2}(\Gamma)$, 
equipped by the product norm
\[
  \|(v,\sigma_\Sigma,\sigma_\Gamma)\|_{\Hcal}^2
  	\isdef \|v\|_{H^1(B)}^2 + \|\sigma_\Sigma\|_{H^{-1/2}(\Sigma)}^2
		+ \|\sigma_\Gamma\|_{H^{-1/2}(\Gamma)}^2.
\]
Further, let $a:\Hcal\times\Hcal\to\Rbbb$,
be the bilinear form defined by
\begin{align}		\label{bilinear form}
  &a\big((v,\sigma_\Sigma,\sigma_\Gamma),(w,\lambda_\Sigma,\lambda_\Gamma)\big)
  	= \displaystyle{\int_B\langle\nabla v,\nabla w\rangle\dif\!\xbfm} \nonumber \\
  &\qquad+ \left(\begin{bmatrix} w \\ \lambda_\Sigma \\ \lambda_\Gamma \end{bmatrix},
	   \begin{bmatrix} \Wcal_{\Sigma\Sigma} & \Kcal_{\Sigma\Sigma}^\star - 1/2 & \Kcal_{\Gamma\Sigma}^\star \\
                       1/2-\Kcal_{\Sigma\Sigma} & \Vcal_{\Sigma\Sigma}             & \Vcal_{\Gamma\Sigma}       \\
		       -\Kcal_{\Sigma\Gamma}    & \Vcal_{\Sigma\Gamma}             & \Vcal_{\Gamma\Gamma}       \end{bmatrix}
	   \begin{bmatrix} v \\ \sigma_\Sigma \\ \sigma_\Gamma \end{bmatrix}\right)_{L^2(\Sigma)\times L^2(\Sigma)\times L^2(\Gamma)}.
\end{align}
For sake of simplicity in representation, we 
omitted the trace operator in expressions like 
$(w,\Wcal_{\Sigma\Sigma}v)_{L^2(\Sigma)}$ etc.

Introducing the linear functional $F:\Hcal\to\Rbbb$,
\[
  F(w,\lambda_\Sigma,\lambda_\Gamma) = (f,w)_{L^2(B)}
  	+ \left(\begin{bmatrix} w \\ \lambda_\Sigma \\ \lambda_\Gamma \end{bmatrix},
	  \begin{bmatrix} -\Wcal_{\Gamma\Sigma} \\
	                  \Kcal_{\Gamma\Sigma}  \\
	                  \Kcal_{\Gamma\Gamma}-1/2\end{bmatrix} 
			  	\tilde{g}\right)_{L^2(\Sigma)\times L^2(\Sigma)\times L^2(\Gamma)},
\]
the variational formulation is given by: Seek 
$(\tilde{u},\sigma_\Sigma,\sigma_\Gamma)\in\Hcal$ 
such that
\begin{equation}	\label{variational formulation}
  a\big((\tilde{u},\sigma_\Sigma,\sigma_\Gamma),(w,\lambda_\Sigma,\lambda_\Gamma)\big) 
  	= F(w,\lambda_\Sigma,\lambda_\Gamma)
\end{equation}
for all $(w,\lambda_\Sigma,\lambda_\Gamma)\in\Hcal$.
In accordance with \cite{EH06}, the variational formulation 
\eqref{variational formulation} admits a unique solution 
$(\tilde{u},\sigma_\Sigma,\sigma_\Gamma)\in\Hcal$ 
for all $F\in\Hcal'$, provided that $D$ has a conformal 
radius which is smaller than one if $d=2$.

\subsection{Galerkin discretization}
Since the variational formulation is stable without further 
restrictions, the discretization is along the lines of \cite{HPPS02}.
We first introduce a uniform triangulation of $B$ which in 
turn induces a uniform triangulation of $\Sigma$. Moreover, 
we introduce a uniform triangulation of the free boundary 
$\Gamma$, which we suppose to have the same mesh 
size as the triangulation of the domain $B$. For the FEM 
part, we consider continuous, piecewise linear ansatz
functions $\{\phi_k^B:k\in\Laplace^B\}$ with respect to the 
given domain mesh. For the BEM part, we employ piecewise 
constant ansatz functions $\{\psi_k^\Phi:k\in\nabla^\Phi\}$ on 
the respective triangulations of the boundaries $\Phi\in\{\Sigma,
\Gamma\}$. 

For sake of simplicity in representation, we set $\phi_k^\Sigma
\isdef \phi_k^B|_\Sigma$ for all $k\in\Laplace^B$. Note that most of these 
functions vanish except for those with nonzero trace which coincide 
with continuous, piecewise linear ansatz functions on $\Sigma$. 
Finally, we shall introduce the set of continuous, piecewise linear 
ansatz functions on the triangulation of $\Gamma$, which 
we denote by $\{\phi_k^\Gamma:k\in\Laplace^\Gamma\}$, 
where $|\Laplace^\Gamma|\sim |\nabla^\Gamma|$. 

Then, introducing the system matrices
\begin{alignat}{3}		\label{matrices}
  &{\bf A} = \Big[(\nabla\phi_{k'}^B,\phi_k^B)_{L^2(B)}\Big]_{k,k'},
  &&\qquad {\bf W}_{\Phi\Psi} = \Big[(\Wcal_{\Phi\Psi}\phi_{k'}^\Phi,\phi_k^\Psi)_{L^2(\Psi)}\Big]_{k,k'},\nonumber \\
  &{\bf B}_{\Phi} = \Big[\textstyle{\frac{1}{2}}(\phi_{k'}^\Phi,\psi_{j,k}^\Phi)_{L^2(\Phi)}\Big]_{k,k'},
  &&\qquad {\bf K}_{\Phi\Psi} = \Big[(\Kcal_{\Phi\Psi}\phi_{k'}^\Phi,\phi_k^\Psi)_{L^2(\Psi)}\Big]_{k,k'},\\
  &{\bf G}_{\Phi} = \Big[(\phi_{k'}^\Phi,\phi_{j,k}^\Phi)_{L^2(\Phi)}\Big]_{k,k'},
  &&\qquad \Vbfm_{\Phi\Psi} = \Big[(\Vcal_{\Phi\Psi}\phi_{k'}^{\Phi},\phi_{j,k}^{\Psi})_{L^2(\Psi)}\Big]_{k,k'},\nonumber
\end{alignat}
where again $\Phi,\Psi\in\{\Sigma,\Gamma\}$, and the data vector
\[
  {\bf g} = \Big[(\tilde g,\phi_k^\Gamma)_{L^2(\Gamma)}\Big]_k,
\]
we obtain the following linear system of equations
\begin{equation}		     \label{linear system}
  \begin{bmatrix} {\bf A} + {\bf W}_{\Sigma\Sigma} & {\bf K}_{\Sigma\Sigma}^T - {\bf B}_\Sigma^T &  {\bf K}_{\Sigma\Gamma}^T   \\
                  {\bf B}_\Sigma - {\bf K}_{\Sigma\Sigma} & \Vbfm_{\Sigma\Sigma} & \Vbfm_{\Gamma\Sigma} \\
                  -{\bf K}_{\Sigma\Gamma}          & \Vbfm_{\Sigma\Gamma} & \Vbfm_{\Gamma\Gamma} \end{bmatrix}
  \begin{bmatrix}{\bf u} \\ {\boldsymbol \sigma}_\Sigma \\ {\boldsymbol \sigma}_\Gamma \end{bmatrix}
  = \begin{bmatrix} -{\bf W}_{\Sigma\Sigma} \\ {\bf K}_{\Gamma\Sigma} \\ {\bf K}_{\Gamma\Gamma}-{\bf B}_\Gamma \end{bmatrix} 
  			{\bf G}_\Gamma^{-1}{\bf g}.
\end{equation}
We mention that ${\bf G}_\Gamma^{-1}{\bf g}$ corresponds 
to the $L^2(\Gamma)$-orthogonal projection of the given Dirichlet 
data $\tilde{g}\in H^{1/2}(\Gamma)$ onto the space of the continuous,
piecewise linear ansatz functions on $\Gamma$. That way, we can
also apply fast boundary element techniques to the boundary integral 
operators on the right hand side of the system \eqref{linear system} 
of linear equations.

The present discretization yields the following error 
estimate, see \cite{EH06}.

\begin{proposition}\label{propo:conv}
Let $h$ denote the mesh size of the triangulations of $B$ 
and $\Gamma$, respectively. We denote the solution of 
\eqref{variational formulation} by $(\tilde u,\sigma_\Sigma,
\sigma_\Gamma)$ and the Galerkin solution by $(\tilde u_h,
\sigma_{\Sigma,h},\sigma_{\Gamma,h})$, respectively. 
Then, we have the error estimate 
\begin{align*}
  &\|(\tilde{u},\sigma_\Sigma,\sigma_\Gamma)
  	-(\tilde{u}_h,\sigma_{\Sigma,h},\sigma_{\Gamma,h})
  		\|_{H^1(B)\times H^{-1/2}(\Sigma)\times H^{-1/2}(\Gamma)} \\
	&\hspace*{3cm}\lesssim h \big(\|(\tilde{u},
		\sigma_\Sigma,\sigma_\Gamma)\|_{H^2(B)\times H^{1/2}(\Sigma)\times H^{1/2}(\Gamma)}
\end{align*}
uniformly in $h$.
\end{proposition}

\subsection{Multilevel based solution of the coupling formulation}\label{section:MLCoupling}
We shall encounter some issues on the efficient multilevel 
based solution of the system \eqref{linear system} of linear 
equations. The complexity is governed by the BEM part 
since the boundary element matrices are densely populated.
Following \cite{HPPS02,HPPS03}, we apply wavelet matrix 
compression to reduce this complexity such that the 
over-all complexity is governed by the FEM part. On 
the other hand, according to \cite{HPPS03,HK96}, the
Bramble-Pasciak-CG (see \cite{BP88}) provides an efficient
and robust iterative solver for the above saddle point system. 
Combining a nested iteration with the BPX preconditioner 
(see \cite{BPX90}) for the FEM part and a wavelet 
preconditioning (see \cite{DK92,Schneider}) for the BEM part, 
we derive an asymptotical optimal solver for the 
above system, see \cite{HPPS03} for the details. We 
refer the reader to \cite{HPPS03} for the details of the
implementation of a similar coupling formulation.

\section{Multilevel quadrature method}\label{section:MLQ}
The crucial idea of the multilevel quadrature to compute
the quantity of interest \eqref{eq:QoI} is to combine an 
appropriate sequence of quadrature rules for the stochastic 
variable with the multilevel discretization in the spatial 
variable. To that end, we first parametrize the quantity 
of interest by using \eqref{eq:KLp} over the cube $\square 
= [-1,1]^{\Nbbb^*}$ and compute
\begin{equation}\label{eq:sg-exp}
  \QoI(u) = \int_\square\Fcal\groupp[\big]{u(\cdot)} \dif\!\Pbbb_\ybfm
  	\approx \sum_{\ell=0}^L {\bf Q}_{L-\ell}\Big(\Fcal\groupp[\big]{u_\ell(\cdot)}
		-\Fcal\groupp[\big]{u_{\ell-1}(\cdot)}\Big),
\end{equation}
where $u_{\ell-1} := 0$. Herein, for the spatial approximation, 
we shall use the multilevel representation from Subsection
\ref{section:MLCoupling} to compute the Galerkin solution 
$u_\ell\in H^1(B)$ on level $\ell$ that corresponds to the 
step size $h_\ell = 2^{-\ell}$. For the approximation in the 
stochastic variable \({\bf y}\), we shall thus provide a sequence
of quadrature formulae $\{{\bf Q}_\ell\}$ for the integral
\[
  \int_{\square} v(\xbfm,\ybfm)\dif\!\Pbbb_\ybfm
\]
of the form
\[
  {\bf Q}_\ell v = \sum_{i=1}^{N_\ell}\boldsymbol\omega_{\ell,i} 
  	v(\cdot,\boldsymbol\xi_{\ell,i}).
\]
For our purposes, we assume that the number of points $N_\ell$ 
of the quadrature formula ${\bf Q}_\ell$ is chosen such that the 
corresponding accuracy is
\begin{equation}	\label{eq:Q2}
  \varepsilon_\ell= 2^{-\ell}, \quad \ell = 0,1,\ldots,L.
\end{equation}
Since the multilevel quadrature can be interpreted as a sparse-grid 
approximation, cf.~\cite{HPS13}, it is known that mixed regularity 
results of the integrand have to be provided as derived in Section
\ref{section:regularity}, compare \cite{DKLNS14,GHP15,HPS13} for
example. Since the mapping $u:\square\to H^{\tau+1}(B)$ is 
analytic, we can especially apply the quasi-Monte Carlo method, 
the Gaussian quadrature, or the sparse grid quadrature, see 
e.g.~\cite{GHP15}. Especially, in case of $H^2$-regularity 
($\tau = 1$) and $\Fcal = u|_B$, i.e., $\QoI(u) = \Mean(u|_B)$, 
we obtain then the error estimate
\begin{equation}\label{eq:concergenceH1}
  \Bigg\|\Mean(u)-\sum_{\ell=0}^L {\bf Q}_{L-\ell}\Big(u_\ell(\cdot)
		-u_{\ell-1}(\cdot)\Big)\Bigg\|_{H^1(B)} = \mathcal{O}(h_L).
\end{equation}
Notice that the computational complexity of the multilevel 
quadrature \eqref{eq:sg-exp} is considerably reduced compared 
to a standard single-level quadrature method which has the same 
accuracy, see e.g.\ \cite{BSZ11,CGST11,HPS13}.

\begin{remark}
By choosing the accuracy of the quadrature in accordance with 
$\varepsilon_\ell = 4^{-\ell}$ for $\ell=0,1,\ldots,L$ instead of 
\eqref{eq:Q2}, the application of the Aubin-Nitsche trick in 
Proposition~\ref{propo:conv} implies the $L^2$-error estimate
\begin{equation}\label{eq:concergenceL2}
  \Bigg\|\Mean(u)-\sum_{\ell=0}^L {\bf Q}_{L-\ell}\Big(u_\ell(\cdot)
		-u_{\ell-1}(\cdot)\Big)\Bigg\|_{L^2(B)} = \mathcal{O}(h_L^2).
\end{equation}
\end{remark}

\section{Numerical results}\label{section:results}
In our numerical example, we consider the reference
domain $D$ to be the ellipse with semi-axis 0.6 
and 0.4. We represent its boundary by $\gamma_{\refd}:
[0,2\pi)\to\partial D$ in polar coordinates and perturb 
this parametrization in accordance with
\[
  \boldsymbol\gamma(\omega,\phi) = \boldsymbol\gamma_{\refd}(\phi) 
  	+ \varepsilon\sum_{k=0}^\infty w_k\big\{ y_{-k}[\omega]\sin(k\phi) + y_k[\omega]\cos(k\phi)\big\}
\]
where $y_k\in(-0.5,0.5)$ for all $k\in\mathbb{Z}$ and
$\varepsilon = 0.05$. The weights $w_k$ are chosen as
$w_k = 1$ for all $|k|\le 5$ and $w_k = (k-5)^{-6}$ for
all $|k|>5$. Hence, we have the decay $\gamma_k\sim k^{-4}$
for the choice $\tau = 1$, which is sufficient for applying the
quasi Monte Carlo method based on the Halton sequence.
In practice, we set all $w_k$ to zero if $|k|>64$ which 
corresponds to a dimension truncation after 129 dimensions.
The random parame\-tri\-zation $\boldsymbol\gamma[\omega]$ 
induces the random domain $\Dfrak[\omega]$. The 
fixed subset $B\subset D$ is given as the ball of 
radius 0.2, centered in the origin. For an illustration of 
four random draws, see Figure~\ref{fig:sample}.

\begin{figure}
\begin{center}
\includegraphics[width=0.80\textwidth]{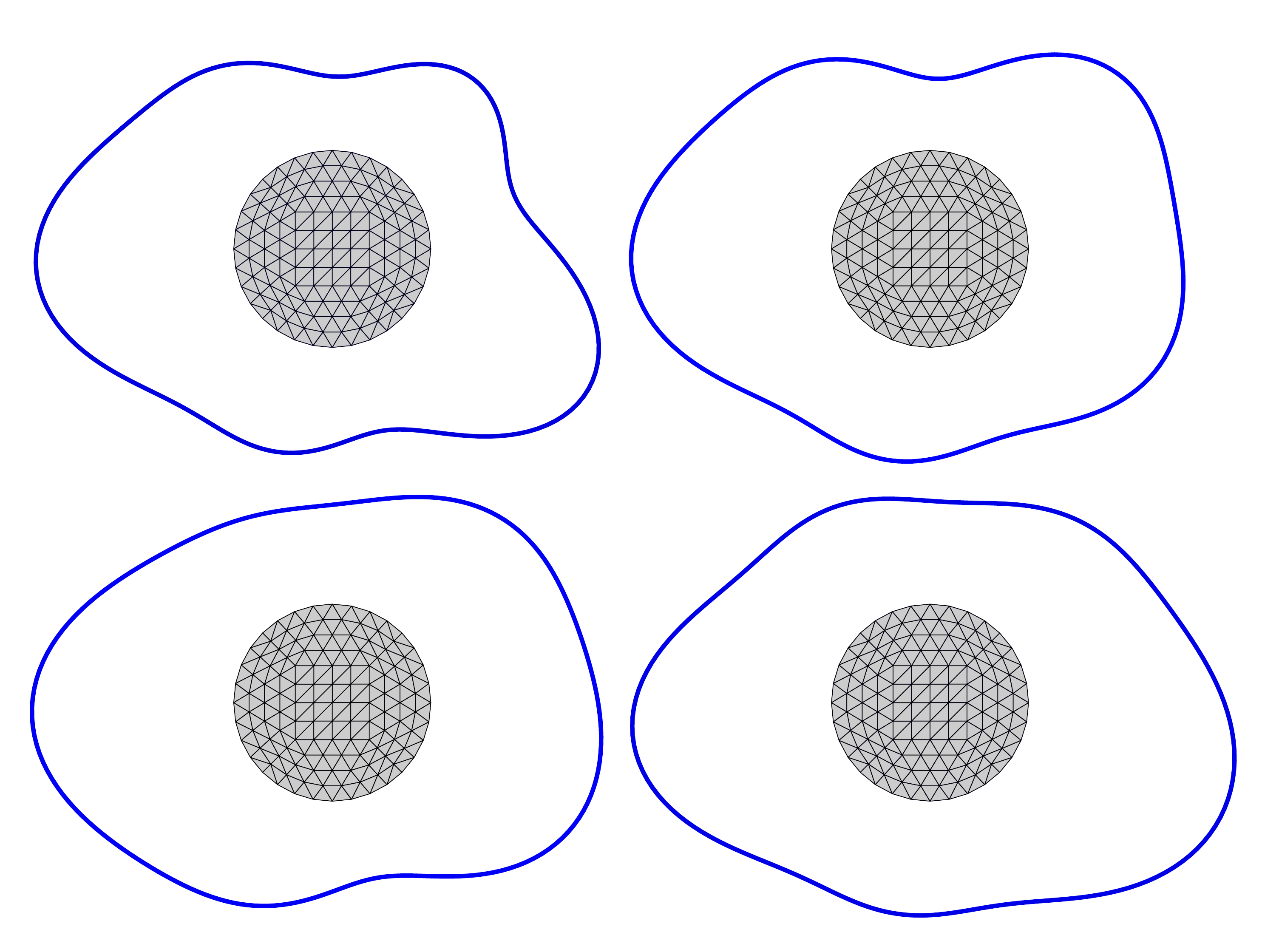}
\caption{\label{fig:sample}Four samples of the random domain
with finite element triangulation of $B$ on refinement level 2.}
\end{center}
\end{figure}

On the random domain $\Dfrak[\omega]$, let the Poisson 
equation 
\[
  -\Laplace u [\omega] = 1\ \text{in $\Dfrak[\omega]$},
  \quad u [\omega] = 0\ \text{on $\partial \Dfrak[\omega]$}
\]
be given. A suitable Newton potential is analytically 
given by $\Ncal_f = -(x_1^2+x_2^2)/4$. We consider the 
$L^2$-tracking type functional
\[
  \QoI(u) = \mathbb{E}\bigg[\frac{1}{2}\int_B |u[\omega]-\overline{u}|^2 \dif\!\xbfm\bigg]
\]
as quantity of interest, where $\overline{u}$ is a given function.
The coarse triangulation of $B$, based on Zl\'amal's curved finite 
elements~\cite{Zenisek}, consists of 14 curved triangles on the 
coarse grid, which are then uniformly refined to get the triangulation 
on the finer grids. The 14 triangles correspond to eight piecewise 
linear and constant boundary elements each on the boundary 
$\partial B$. At the boundary $\partial D$, we likewise consider 
eight piecewise linear and constant boundary elements each 
on level 0. When applying uniform refinement, we arrive at 
the numbers of degrees of freedom in the finite and boundary
element spaces found in Table~\ref{tab:dof}.

\begin{table}
\begin{center}
  \begin{tabular}{rrr}
    \toprule
    level & finite elements & boundary elements \\
    \midrule
    1 & 37 & 32 \\
    2 & 129 & 64 \\
    3 & 481 & 128 \\
    4 & 1857 & 256 \\
    5 & 7297 & 512 \\
    6 & 28929 & 1024 \\
    7 & 115201 & 2048 \\
    8 & 459777 & 4096 \\
    \bottomrule \\
  \end{tabular}
\caption{\label{tab:dof}Number of the degrees of freedom 
of the finite element method and the boundary element method.}
\end{center}
\end{table}

In order to compute the quantity of interest, we will employ the 
quasi-Monte Carlo method based on the Halton sequence, see 
\cite{H60} for example. Since the exact solution is unknown, 
we compute first the quantity of interest on the spatial discretization 
level 8 by using 100\,000 Halton points. Next, we compute the solution 
by the multilevel quasi-Monte Carlo method. Namely, for the multilevel
quasi-Monte Carlo method on level $L$, we apply $N_\ell = 2^{L-\ell} 
N_L$ and $N_\ell = 4^{L-\ell} N_L$ Halton points, respectively, on the
coarser levels $0\le\ell\le L$, where we choose $N_L = 10, 20, 40$ fine 
grid samples.

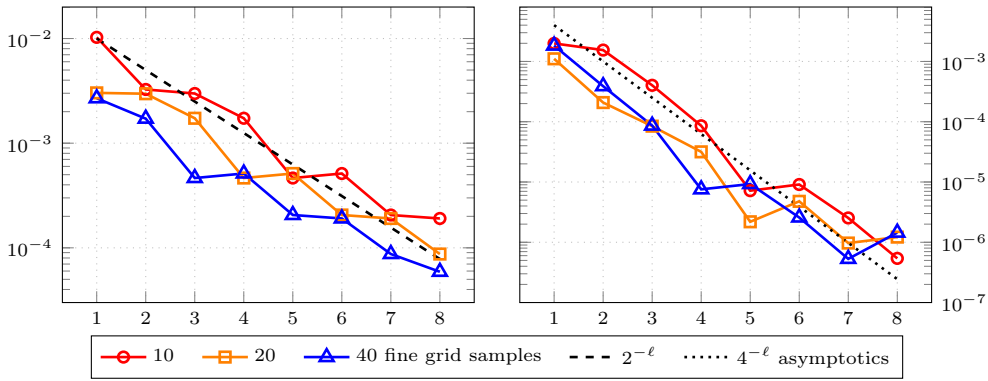
\begin{figure}
\centering
\begin{tikzpicture}[baseline]
  \begin{semilogyaxis}[
      anchor=south east,
      at={(-0.3cm,0cm)},
      yticklabel pos=left,
      xmajorgrids=true,
      ymajorgrids=true,
      ymax=2e-2,
      ymin=3e-5,
      xtick={1,2,...,8},
      max space between ticks=20,
    ]
    \addplot[red, mark=o, mark size=2pt, line width=1pt] table [x=l,y=f10] {
l f10 f20 f40
   1  1.0263e-02   3.0248e-03   2.6850e-03
   2  3.2602e-03   2.9691e-03   1.7151e-03
   3  2.9822e-03   1.7305e-03   4.6444e-04
   4  1.7306e-03   4.6459e-04   5.1323e-04
   5  4.6449e-04   5.1311e-04   2.0593e-04
   6  5.1304e-04   2.0587e-04   1.9084e-04
   7  2.0584e-04   1.9081e-04   8.7246e-05
   8  1.9082e-04   8.7252e-05   5.9037e-05 
   };
    \addplot[orange, mark=square, mark size=2pt, line width=1pt] table [x=l,y=f20] {
l f10 f20 f40
   1  1.0263e-02   3.0248e-03   2.6850e-03
   2  3.2602e-03   2.9691e-03   1.7151e-03
   3  2.9822e-03   1.7305e-03   4.6444e-04
   4  1.7306e-03   4.6459e-04   5.1323e-04
   5  4.6449e-04   5.1311e-04   2.0593e-04
   6  5.1304e-04   2.0587e-04   1.9084e-04
   7  2.0584e-04   1.9081e-04   8.7246e-05
   8  1.9082e-04   8.7252e-05   5.9037e-05 
    };
    \addplot[blue, mark=triangle, mark size=3pt, line width=1pt] table [x=l,y=f40] {
l f10 f20 f40
   1  1.0263e-02   3.0248e-03   2.6850e-03
   2  3.2602e-03   2.9691e-03   1.7151e-03
   3  2.9822e-03   1.7305e-03   4.6444e-04
   4  1.7306e-03   4.6459e-04   5.1323e-04
   5  4.6449e-04   5.1311e-04   2.0593e-04
   6  5.1304e-04   2.0587e-04   1.9084e-04
   7  2.0584e-04   1.9081e-04   8.7246e-05
   8  1.9082e-04   8.7252e-05   5.9037e-05 
    };
    \addplot[black, dashed, line width=1pt] table [x=l,y expr=0.02*2^(-\thisrow{l})] {
l f10 f20 f40
   1  1.0263e-02   3.0248e-03   2.6850e-03
   2  3.2602e-03   2.9691e-03   1.7151e-03
   3  2.9822e-03   1.7305e-03   4.6444e-04
   4  1.7306e-03   4.6459e-04   5.1323e-04
   5  4.6449e-04   5.1311e-04   2.0593e-04
   6  5.1304e-04   2.0587e-04   1.9084e-04
   7  2.0584e-04   1.9081e-04   8.7246e-05
   8  1.9082e-04   8.7252e-05   5.9037e-05 
    }; \label{asy}
  \end{semilogyaxis}
  \begin{semilogyaxis}[
      anchor=south west,
      at={(0.3cm,0cm)},
      yticklabel pos=right,
      xmajorgrids=true,
      ymajorgrids=true,
      ymax=8e-3,
      ymin=1e-7,
      xtick={1,2,...,8},
      max space between ticks=20,
      legend columns =-1,
      legend to name=bothlegend,
    ]
    \addplot[red, mark=o, mark size=2pt, line width=1pt] table [x=l,y=f10] {
l f10 f20 f40
   1  1.9909e-03   1.1068e-03   1.8321e-03
   2  1.5449e-03   2.0802e-04   3.8881e-04
   3  4.0202e-04   8.4404e-05   8.6239e-05
   4  8.5680e-05   3.1638e-05   7.5557e-06
   5  7.1797e-06   2.1875e-06   9.2450e-06
   6  9.0885e-06   4.7805e-06   2.5761e-06
   7  2.5328e-06   9.7100e-07   5.3046e-07
   8  5.4077e-07   1.2185e-06   1.4453e-06
    };
    \addlegendentry{$10$}
    \addlegendimage{empty legend}\addlegendentry{\hspace*{0.5em}}
    \addplot[orange, mark=square, mark size=2pt, line width=1pt] table [x=l,y=f20] {
l f10 f20 f40
   1  1.9909e-03   1.1068e-03   1.8321e-03
   2  1.5449e-03   2.0802e-04   3.8881e-04
   3  4.0202e-04   8.4404e-05   8.6239e-05
   4  8.5680e-05   3.1638e-05   7.5557e-06
   5  7.1797e-06   2.1875e-06   9.2450e-06
   6  9.0885e-06   4.7805e-06   2.5761e-06
   7  2.5328e-06   9.7100e-07   5.3046e-07
   8  5.4077e-07   1.2185e-06   1.4453e-06
    };
    \addlegendentry{$20$}
    \addlegendimage{empty legend}\addlegendentry{\hspace*{0.5em}}
    \addplot[blue, mark=triangle, mark size=3pt, line width=1pt] table [x=l,y=f40] {
l f10 f20 f40
   1  1.9909e-03   1.1068e-03   1.8321e-03
   2  1.5449e-03   2.0802e-04   3.8881e-04
   3  4.0202e-04   8.4404e-05   8.6239e-05
   4  8.5680e-05   3.1638e-05   7.5557e-06
   5  7.1797e-06   2.1875e-06   9.2450e-06
   6  9.0885e-06   4.7805e-06   2.5761e-06
   7  2.5328e-06   9.7100e-07   5.3046e-07
   8  5.4077e-07   1.2185e-06   1.4453e-06
  };
    \addlegendentry{$40$ fine grid samples}
    \addlegendimage{empty legend}\addlegendentry{\hspace*{0.5em}}
    \addlegendimage{/pgfplots/refstyle=asy}\addlegendentry{$2^{-\ell}$}
    \addlegendimage{empty legend}\addlegendentry{\hspace*{0.5em}}
    \addplot[black, dotted, line width=1pt] table [x=l,y expr=0.016*4^(-\thisrow{l})] {
l f10 f20 f40
   1  1.9909e-03   1.1068e-03   1.8321e-03
   2  1.5449e-03   2.0802e-04   3.8881e-04
   3  4.0202e-04   8.4404e-05   8.6239e-05
   4  8.5680e-05   3.1638e-05   7.5557e-06
   5  7.1797e-06   2.1875e-06   9.2450e-06
   6  9.0885e-06   4.7805e-06   2.5761e-06
   7  2.5328e-06   9.7100e-07   5.3046e-07
   8  5.4077e-07   1.2185e-06   1.4453e-06
  };
    \addlegendentry{$4^{-\ell}$ asymptotics}
  \end{semilogyaxis}
\end{tikzpicture}
\ref{bothlegend}
\caption{\label{fig:convergence}Absolute error of the output functional
for different numbers of fine grid samples when increasing the number of
quadrature point linearly (left) and quadratically (right) per level.}
\end{figure}

As it is seen in Figure~\ref{fig:convergence}, we always 
observe the same linear and quadratic convergence rate, respectively,
but with different constants involved. Notice that linear convergence 
is in accordance with \eqref{eq:concergenceH1} while quadratic 
convergence is in accordance with \eqref{eq:concergenceL2}.
Only for the level $\ell=8$ and desired quadratic convergence, 
we observe a stagnation of the convergence. This issues from
the fact that the reference solution computed by a single-level
quadrature method is not accurate enough.

\section{Conclusion}\label{section:conclusion}
We provided regularity estimates of the solution to
elliptic problems on random domains which allow for
the application of multilevel quadrature methods. In
order to avoid the need to compute either a random domain mapping
or to generate meshes for every domain sample,
we couple finite elements with boundary elements.
It has been shown by numerical experiments that this 
approach is indeed able to exploit the additional regularity 
we have in the underlying problem without causing numerical 
problems on too coarse grids.
%
%
\bibliographystyle{plain}
\bibliography{articles,books}

\begin{thebibliography}{10}

\bibitem{BSZ11}
A.~Barth, C.~Schwab, and N.~Zollinger.
\newblock Multi-level {M}onte {C}arlo finite element method for elliptic {PDE}s
  with stochastic coefficients.
\newblock {\em Numer. Math.}, 119(1):123--161, 2011.

\bibitem{BP88}
J.~Bramble and J.~E. Pasciak.
\newblock Preconditioner technique for indefinite systems resulting from mixed
  approximation of elliptic problems.
\newblock {\em Math. Comput.}, 50:1--17, 1988.

\bibitem{BPX90}
J.~Bramble, J.~E. Pasciak, and J.~Xu.
\newblock Parallel multilevel preconditioners.
\newblock {\em Math. Comput.}, 55:1--22, 1990.

\bibitem{BLN17}
C.~B\v{a}cu\c{t}\v{a}, H.~Li, and V.~Nistor.
\newblock Differential operators on domains with conical points: precise
  uniform regularity estimates.
\newblock {\em Rev. Roumaine de Math. Pures Appl.}, 62(3):383--411, 2017.

\bibitem{CK07}
C.~Canuto and T.~Kozubek.
\newblock A fictitious domain approach to the numerical solution of {PDE}s in
  stochastic domains.
\newblock {\em Numer. Math.}, 107(2):257--293, 2007.

\bibitem{CNT16}
J.~E. Castrillon-Candas, F.~Nobile, and R.~Tempone.
\newblock Analytic regularity and collocation approximation for {PDE}s with
  random domain deformations.
\newblock {\em Comput. Math. Appl.}, 71(6):1173--1197, 2016.

\bibitem{CGST11}
K.~A. Cliffe, M.~B. Giles, R.~Scheichl, and A.~L. Teckentrup.
\newblock Multilevel {M}onte {C}arlo methods and applications to elliptic
  {PDE}s with random coefficients.
\newblock {\em Comput. Vis. Sci.}, 14(1):3--15, 2011.

\bibitem{CS88}
M.~Costabel and E.~P. Stephan.
\newblock Coupling of finite element and boundary element methods for an
  elasto-plastic interface problem.
\newblock {\em SIAM J. Numer. Anal.}, 27:1212--1226, 1988.

\bibitem{DK92}
W.~Dahmen and A.~Kunoth.
\newblock Multilevel preconditioning.
\newblock {\em Numer. Math.}, 63(3):315--344, 1992.

\bibitem{DKLNS14}
J.~Dick, F.~Y. Kuo, Q.~T. Le~Gia, D.~Nuyens, and C.~Schwab.
\newblock Higher order {QMC} {P}etrov--{G}alerkin discretization for parametric
  operator equations.
\newblock {\em SIAM J. Numer. Anal.}, 52(6):2676--2702, 2014.

\bibitem{EH06}
K.~Eppler and H.~Harbrecht.
\newblock Coupling of {FEM} and {BEM} in shape optimization.
\newblock {\em Numer. Math.}, 104(1):47--68, 2006.

\bibitem{G71}
W.~J. Gordon.
\newblock Blending-function methods of bivariate and multivariate interpolation
  and approximation.
\newblock {\em SIAM J. Numer. Anal.}, 8(1):158--177, 1971.

\bibitem{GH73}
W.~J. Gordon and C.~A. Hall.
\newblock Construction of curvilinear co-ordinate systems and applications to
  mesh generation.
\newblock {\em Int. J. Numer. Meth. Engng.}, 7(4):461--477, 1973.

\bibitem{GT82}
W.~J. Gordon and L.~C. Thiel.
\newblock Transfinite mappings and their application to grid generation.
\newblock {\em Appl. Math. Comput.}, 10--11:171--233, 1982.

\bibitem{GHP15}
M.~Griebel, H.~Harbrecht, and M.D. Multerer.
\newblock Multilevel quadrature for elliptic parametric partial differential
  equations in case of polygonal approximations of curved domains.
\newblock {\em ArXiv e-prints arXiv:1509.09058v2}, 2018.
\newblock to appear in SIAM J. Numer. Anal.

\bibitem{Grisvard}
P.~Grisvard.
\newblock {\em Elliptic Problems in Nonsmooth Domains}.
\newblock Classics in Applied Mathematics. Society for Industrial and Applied
  Mathematics, 2011.

\bibitem{H60}
J.~H. Halton.
\newblock On the efficiency of certain quasi-random sequences of points in
  evaluating multi-dimensional integrals.
\newblock {\em Numer. Math.}, 2(1):84--90, 1960.

\bibitem{H90}
H.~Han.
\newblock A new class of variational formulation for the coupling of finite and
  boundary element methods.
\newblock {\em J. Comput. Math.}, 8(3):223--232, 1990.

\bibitem{HPPS02}
H.~Harbrecht, F.~Paiva, C.~P\'erez, and R.~Schneider.
\newblock {B}iorthogonal wavelet approximation for the coupling of {FEM-BEM}.
\newblock {\em Numer. Math.}, 92:325--356, 2002.

\bibitem{HPPS03}
H.~Harbrecht, F.~Paiva, C.~P\'erez, and R.~Schneider.
\newblock {W}avelet preconditioning for the coupling of {FEM-BEM}.
\newblock {\em Numer. Linear Algebra Appl.}, 3:197--222, 2003.

\bibitem{HPS13}
H.~Harbrecht, M.~Peters, and M.~Siebenmorgen.
\newblock On multilevel quadrature for elliptic stochastic partial differential
  equations.
\newblock {\em Sparse Grids and Applications}, 88:161--179, 2013.

\bibitem{HPS16}
H.~Harbrecht, M.~Peters, and M.~Siebenmorgen.
\newblock Analysis of the domain mapping method for elliptic diffusion problems
  on random domains.
\newblock {\em Numer. Math.}, 134(4):823--856, 2016.

\bibitem{HS19}
H.~Harbrecht and M.~Schmidlin.
\newblock Multilevel methods for uncertainty quantification of elliptic {PDE}s
  with random anisotropic diffusion.
\newblock {\em Stoch. Partial Differ. Equ. Anal. Comput.}, 2019.

\bibitem{HSS08}
H.~Harbrecht, R.~Schneider, and C.~Schwab.
\newblock Sparse second moment analysis for elliptic problems in stochastic
  domains.
\newblock {\em Numer. Math.}, 109(3):385--414, 2008.

\bibitem{HK96}
B.~Heise and M.~Kuhn.
\newblock Parallel solvers for linear and nonlinear exterior magnetic field
  problems based upon coupled {FE}/{BE} formulations.
\newblock {\em Computing}, 56:237--258, 1996.

\bibitem{HillePhillips}
E.~Hille and R.~S. Phillips.
\newblock {\em Functional Analysis and Semi-Groups}, volume~31 of {\em Amer.
  Math. Soc. Collog. Publ.}
\newblock American Mathematical Society, Providence, 1957.

\bibitem{MNK11}
P.~S. Mohan, P.~B. Nair, and A.~J. Keane.
\newblock Stochastic projection schemes for deterministic linear elliptic
  partial differential equations on random domains.
\newblock {\em Int. J. Numer. Meth. Eng.}, 85(7):874--895, 2011.

\bibitem{Schneider}
R.~Schneider.
\newblock {\em {M}ultiskalen- und {W}avelet-{M}atrixkompression:
  {A}nalysisbasierte {M}ethoden zur {L}{\"o}sung gro{\ss{}}er vollbesetzter
  {G}leichungssyteme}.
\newblock B.~G. Teubner, Stuttgart, 1998.

\bibitem{TX06}
D.M. Tartakovsky and D.~Xiu.
\newblock Stochastic analysis of transport in tubes with rough walls.
\newblock {\em J. Comput. Phys.}, 217(1):248--259, 2006.

\bibitem{XT06}
D.~Xiu and D.~M. Tartakovsky.
\newblock Numerical methods for differential equations in random domains.
\newblock {\em SIAM J. Sci. Comput.}, 28(3):1167--1185, 2006.

\bibitem{Zenisek}
A.~Zenisek.
\newblock {\em Nonlinear {Elliptic} and {Evolution} {Problems} and {Their}
  {Finite} {Element} {Approximation}}.
\newblock Academic Press, London, 1990.

\end{thebibliography}
\end{document}